\documentclass[11pt]{article}
\usepackage{authblk}
\usepackage{amsmath,amssymb,stmaryrd}
\usepackage{enumitem}
\usepackage{verbatim}
\usepackage{color}
\usepackage[margin=1in]{geometry}
\usepackage{setspace}

\newcommand{\limp}{\longrightarrow}
\renewcommand{\lnot}{\neg\,}

\newcommand{\Fmbf}{\mathbf{Fm}}
\newcommand{\Fm}{\mathit{Fm}}
\newcommand{\inyields}{\vdash_\mathrm{IPC}}
\newcommand{\parayields}{\vdash_\mathrm{HIPWK}}
\newcommand{\inmodels}{\models_\mathrm{IPC}}
\newcommand{\paramodels}{\models_{\mathrm{IPWK}}}
\newcommand{\pryields}{\vdash_\mathrm{HPRL}}
\newcommand{\ppryields}{\vdash_{\HPPRL}}
\newcommand{\prmodels}{\models_\mathrm{PRL}}
\newcommand{\pprmodels}{\models_{\mathrm{PPRL}}}
\newcommand{\lvari}{\vdash^l}
\newcommand{\lrri}{\vdash^{re}}
\newcommand{\var}{\mathrm{var}}
\newcommand{\HPPRL}{\mathrm{HPRL}^{re}}
\newcommand{\val}{v^\#}
\newcommand{\sharpE}[1]{\mathbf{#1}^\#}
\newcommand{\NN}{\mathbb{N}}
\newcommand{\pow}{\mathcal{P}}
\newcommand{\Pl}{\mathcal{P}_{\hbox{\scriptsize{\textit{\l{}}}}}}
\newcommand{\lps}{\mathbb{L}\mathrm{PS}_3}
\newcommand{\rml}{\mathrm{RM}_3}

\usepackage{amsthm}
\theoremstyle{definition}
\newtheorem{thm}{Theorem}[section]
\newtheorem{lem}[thm]{Lemma}
\newtheorem{cor}[thm]{Corollary}
\newtheorem{dfn}[thm]{Definition}
\newtheorem{rem}[thm]{Remark}

\usepackage[backref=page,colorlinks=true,allcolors=blue]{hyperref}

  \title{Restricted Rules of Inference and Paraconsistency\footnote{The final version of this paper has been published online in Logic Journal of the IGPL (\href{https://doi.org/10.1093/jigpal/jzab019}{https://doi.org/10.1093/jigpal/jzab019}).}}
  \author[1]{Sankha S. Basu}
  \date{July 12, 2021}
  \affil[1]{Department of Mathematics\\
  Indraprastha Institute of Information Technology-Delhi\\
  New Delhi, India. (\href{mailto:sankha@iiitd.ac.in}{sankha@iiitd.ac.in})}

  \author[2]{Mihir K.\ Chakraborty}
  \affil[2]{School of Cognitive Science\\Jadavpur University\\ Kolkata, India. (\href{mailto:mihirc4@gmail.com}{mihirc4@gmail.com})}

\begin{document}
\maketitle
\begin{abstract}
  \addcontentsline{toc}{section}{Abstract}
  \noindent In this paper, we study two companions of a logic, viz., the left variable inclusion companion and the restricted rules companion, their nature and interrelations, especially in connection with paraconsistency. A sufficient condition for the two companions to coincide has also been proved. Two new logical systems - Intuitionistic Paraconsistent Weak Kleene logic (IPWK) and Paraconsistent Pre-Rough logic (PPRL) - are presented here as examples of logics of left variable inclusion. IPWK is the left variable inclusion companion of Intuitionistic Propositional logic (IPC) and is also the restricted rules companion of it. PPRL, on the other hand, is the left variable inclusion companion of Pre-Rough logic (PRL) but differs from the restricted rules companion of it. We have discussed algebraic semantics for these logics in terms of P\l{}onka sums. This amounts to introducing a contaminating truth value, intended to denote a state of indeterminacy.
\end{abstract}
\textbf{Keywords:}
Paraconsistency, Left variable inclusion, P\l{}onka sums, Contaminating element, Intuitionistic logic, Pre-rough logic.

\tableofcontents
%\doublespacing
\section{Introduction}\label{sec:introduction}

\subsection{Paraconsistency}\label{subsec:paraconsistency}
A \emph{paraconsistent} logic is commonly understood as a logic where an \emph{inconsistency} does not lead to \emph{triviality}. Here inconsistency refers to negation inconsistency, i.e. the presence of a contradiction, a pair of formulas such that one is the negation of the other. On the other hand, the entailment of every formula is called triviality. In other words, a paraconsistent logic is a logic where it is not always possible to derive every formula from a contradiction. Thus paraconsistent logics are described as those which give rise to inconsistent but non-trivial \emph{theories} (sets of formulas closed under logical consequence). Classical logic, and many non-classical logics, such as intuitionistic logic, fail in this because of the so-called principle of \emph{`explosion'}, which says that for any formula $\alpha$,
\[
\{\alpha,\neg\alpha\}\vdash\beta,
\]
where $\neg\alpha$ denotes `not $\alpha$', $\vdash$ is the logical consequence relation and $\beta$ is any formula, even if it is unrelated to $\alpha$. This rule is also commonly known as ECQ or \emph{ex contradictione sequitur quodlibet} (Latin for ``from a contradiction, anything follows''). If the logical consequence relation obeys the Tarskian transitivity  condition (see Definition \ref{dfn:consrel}), then it is paraconsistent iff it is not explosive, that is ECQ fails in it \cite{Priest1984}.

We call the following rule of explosion $\land$-ECQ. For any formulas $\alpha,\beta$, 
\[
\alpha\land\neg\alpha\vdash\beta,
\]
where $\land$ denotes conjunction. Paraconsistency can also be described as the failure of $\land$-ECQ. Now, of course, if a logic that obeys the Tarskian conditions, has the following rules of introduction and elimination of conjunction,
\[
\begin{array}{ll}
     \land I.&\{\alpha,\beta\}\vdash\alpha\land\beta,\\
     \land E.&\alpha\land\beta\vdash\alpha,\,\alpha\land\beta\vdash\beta,
\end{array}
\]
then ECQ and $\land$-ECQ are equivalent and hence the failure of one leads to the failure of the other. There is, however, a group of paraconsistent logics, called \emph{non-adjunctive} logics \cite{Priest1984} that lack $\land I$. On the other hand, the Paraconsistent Weak Kleene Logic (PWK) is an example of a paraconsistent logic that lacks $\land E$ \cite{CiuniCarrara2016}. In this paper we have a few more examples of logics without $\land E$ (see Subsections \ref{subsec:ipwk} and \ref{subsec:pprl} for specific examples and Section \ref{sec:lvari} for a class of logics). These logics are paraconsistent in the sense that ECQ fails in these. 

Another principle linked with the study of paraconsistency is the \emph{law of non-contradiction} (LNC), the syntactice version of which can be stated as follows: for any formula $\alpha$, $\neg(\alpha\land\neg\alpha)$ is a theorem, i.e. $\vdash\neg(\alpha\land\neg\alpha)$. Classical and intuitionistic logics validate explosion (both ECQ and $\land$-ECQ) as well as LNC. However, there are logics, such as the three-valued logic of \L{}ukasiewicz, $L_3$, where ECQ holds, but LNC fails\footnote{It may be noted that $L_3$ is sometimes presented with two conjunctions - a lattice theoretic $\land$ and a monoidal $\odot$ - and LNC fails for $\land$, i.e. there exists formulas $\alpha$ such that $\not\vdash\neg(\alpha\land\neg\alpha)$ in $L_3$.}. On the other hand, there are also examples of logics, such as Priest's Logic of Paradox (LP), where ECQ fails but LNC holds \cite{Priest1979}. Thus ECQ and LNC are independent of each other and failure of LNC alone, i.e. the existence of some formula $\alpha$ such that $\not\vdash\neg(\alpha\land\neg\alpha)$, cannot prevent explosion. Since paraconsistency is understood as the failure of explosion, while LP is regarded as a paraconsistent system, $L_3$ is not. Nevertheless, the failure of LNC marks an important departure from the classical way of handling a contradiction. Paraconsistent logics where both ECQ and LNC fail are sometimes called \emph{strong} paraconsistent logics \cite{Beziau2014}.

A still stronger notion, called \emph{dialetheism} is the acceptance of true contradictions or \emph{dialethias}. It can be easily shown that nontrivial logical systems that admit dialethias, or dialetheic logics are necessarily paraconsistent, i.e. ECQ must fail in these systems. Clearly, LNC fails in these systems as well. Thus dialetheic paraconsistent logics are strongly paraconsistent. The converse of this is, however, not the case.

To summarize, we have at least four possible flavors of paraconsistency, viz.,
\begin{enumerate}[label=(\roman*)]
    \item Paraconsistency owing to failure of ECQ
    \item Paraconsistency owing to failure of $\land$-ECQ
    \item Strong paraconsistency: failure of ECQ and LNC
    \item Dialetheic paraconsistency: existence of ``true contradictions'' \cite{Priest1984}
\end{enumerate}
It may be noted that (i) and (ii) above (equivalent in case the logic possesses $\land I$ and $\land E$) are sometimes referred to as \emph{weak} paraconsistency. It is clear from the above discussion that paraconsistent logics can be categorized into different classes depending on the classical principles that we choose to forego to obtain them. Comparisons between some well-known paraconsistent systems can be found in \cite{ArieliAvron2015, DuttaChakraborty2011, DuttaChakraborty2015, TarafderChakraborty2015}. More about paraconsistency along with its history and motivations can be found in \cite{PriestRoutley1989, Beziau2000, Mortensen1995}.

In the current paper, we explore the nature of paraconsistency obtained by imposing \emph{left variable inclusion constraints} \cite{Bonzio_2020} on logics that may be non-paraconsistent themselves. A known example of a paraconsistent logic obtained from classical propositional logic in this way is PWK, which has the same set of theorems as classical propositional logic. See \cite{Bonzio2017,CiuniCarrara2016,CiuniCarrara2019} for more on PWK. An algebraic interpretation for PWK, using \emph{P\l{}onka sums} of Boolean algebras, is given in \cite{Bonzio2017}. A Hilbert-style presentation of PWK consisting of the same set of axioms as classical propositional logic, and only one rule of inference which is a restricted version of the classical rule of \emph{modus ponens}, is also given in \cite{Bonzio2017}. Thus the following natural question arises. Is it always the case that Hilbert-style presentations of logics of left variable inclusion can be obtained by keeping the same set of axioms as the original logics and only restricting the rules of inference? We have shown here that the answer to this question is negative.

Two new logics - the \emph{intuitionistic paraconsistent weak Kleene logic} (IPWK) and the \emph{paraconsistent pre-rough logic} (PPRL) - have been presented here as specific examples of paraconsistent systems obtained by imposing left variable inclusion constraints on non-paraconsistent logics. IPWK and PPRL have the same sets of theorems as intuitionistic propositional logic (IPC) and pre-rough logic (PRL), respectively. Algebraic semantics for these logics can be given using P\l{}onka sums of Heyting algebras (in case of IPWK) and pre-rough algebras (in case of PPRL). We note that this amounts to having an additional designated truth-value, which we denote by $\omega$, to indicate a state of `indeterminacy' that is clearly distinct from falsity. This truth-value satisfies the so-called principle of \emph{contamination}, that is, the truth-value of any compound formula is $\omega$, i.e. indeterminate, as soon as the truth-value of any of its components is $\omega$. Such infectious or contaminating truth-values have been used to interpret a variety of phenomena that arise naturally in domains such as linguistics, epistemology and computer science. See \cite{CiuniFergusonSzmuc2019,CiuniFergusonSzmuc2019-2,CiuniCarrara2019} for more detailed discussions on this.

We note that while logics obtained either by imposing left variable inclusion constraints or by restricting the rules of inference always lead to paraconsistent systems, there are other well-known paraconsistent logics that cannot be so described. Moreover, the type of paraconsistency obtained in a paraconsistent system obtained via left variable inclusion differs from one instance to another. For example, we see that LNC holds in IPWK, while it fails in PPRL.

\subsection{Outline of the paper}
Our interest and focus is mainly on paraconsistency, we have hence discussed the topic at length in Subsection \ref{subsec:paraconsistency}. The tools we have used from abstract algebraic logic are discussed in Section \ref{sec:aal}. The definitions of P\l{}onka sums and logics of left variable inclusion, and the connection between these are also included in this section. 

In Section \ref{sec:lvari}, we have defined the \emph{restricted rules companion} of a logic with a Hilbert-style presentation, explored connections between the left variable inclusion and the restricted rules companions of logics. It has also been established here that both companions are paraconsistent. In Section \ref{sec:suffcond-Sr=Sl}, we first give a sufficient condition for the two companion logics to coincide. Then Subsections \ref{subsec:ipwk} and \ref{subsec:pprl} introduce the logics IPWK and PPRL as examples of logics of left variable inclusion, such that, in the first case, the two companions coincide, while in the second, they do not.

Finally, in Section \ref{sec:paracons_other}, a couple of examples of paraconsistent logics that do not belong to the classes of logics of left variable inclusion or restricted rules companion logics, have been included.

\section{Algebraic preliminaries and logics of left variable inclusion}\label{sec:aal}

\subsection{Algebraic preliminaries}
This subsection consists of some standard abstract algebraic notions that we will be using in the rest of the paper. These and more can be found in \cite{FontJansanaPigozzi2003, Font2016}

\begin{dfn}
  A \emph{logical language} is a set of connectives/ logical operators, each with a fixed arity $n\in\NN=\{0,1,2,3,\ldots\}$.

  Given a logical language $\mathcal{L}$ and a countably infinite set $V$ of propositional variables, the formulas are defined inductively in the usual way. The connectives/ operators can be regarded as the operation symbols of an algebraic similarity type and then the formulas are the terms of this similarity type over the set $V$. The resulting algebra of terms of type $\mathcal{L}$ over $V$ is called the \emph{algebra of formulas} or \emph{formula algebra} of type $\mathcal{L}$ over $V$. We will denote this by $\Fmbf$. The underlying set of this algebra is the set of formulas of type $\mathcal{L}$ over $V$, denoted by $\Fm$. The operations of $\Fmbf$ are those that are used for forming complex formulas and are associated with the operators in $\mathcal{L}$.
\end{dfn}

\begin{rem}
The formula algebra of some type $\mathcal{L}$ over a set of variables $V$ has the universal mapping property for the class of all algebras of type $\mathcal{L}$ over $V$, i.e. any function $f:V\to A$, where $A$ is the universe of an algebra $\mathbf{A}$ of type $\mathcal{L}$, can be uniquely extended to a homomorphism from $\Fmbf$ to $\mathbf{A}$. Thus, in fact, $\Fmbf$ is the absolutely free algebra of type $\mathcal{L}$ over $V$ \cite[Chapter II, \S 10]{BurrisSankappanavar1981}.
\end{rem}

The following definition is a variant of Tarski's famous definition of a \emph{finitary consequence operator}\footnote{An English translation of this paper from 1930 can be found in \cite{Tarski1983}.}.

\begin{dfn}\label{dfn:consrel}
Suppose $\Fmbf$ is an algebra of formulas of some type $\mathcal{L}$ over a set of variables $V$, with universe $\Fm$. A relation $\vdash\,\subseteq\pow(\Fm)\times \Fm$ is called a \emph{consequence relation on $\Fm$} if it satisfies the following conditions.
\begin{enumerate}[label=(C\arabic*),series=cons]
    \item For any $\Sigma\subseteq\Fm$, if $\varphi\in\Sigma$, then $\Sigma\vdash\varphi$. (\emph{Reflexivity})
    \item For any $\Sigma,\Delta\subseteq\Fm$ and $\varphi\in\Fm$, if $\Delta\vdash\psi$ for all $\psi\in\Sigma$, and $\Sigma\vdash\varphi$, then $\Delta\vdash\varphi$. (\emph{Transitivity})
\end{enumerate}
The above two conditions imply the following.
\begin{enumerate}[label=(C\arabic*),resume=cons]
    \item For any $\Sigma,\Delta\subseteq\Fm$ and $\varphi\in\Fm$, if $\Sigma\vdash\varphi$ and $\Sigma\subseteq\Delta$, then $\Delta\vdash\varphi$. (\emph{Monotonicity})
\end{enumerate}
\end{dfn}

\L o\'{s} and Suszko, in 1958 \cite{LosSuszko1958}, added the condition of invariance under substitutions to the Tarskian conditions.

  Given a logical language $\mathcal{L}$ and a set of variables $V$, a \emph{substitution} is a function $\sigma:V\to\Fm$, which then extends to a unique endomorphism of the formula algebra $\Fmbf$ of type $\mathcal{L}$ over $V$, via the universal mapping property.

\begin{dfn}\label{dfn:structurality}
  A consequence relation is called \emph{substitution-invariant} provided it satisfies the following condition in addition to (C1) and (C2), and hence also (C3), above.
  \begin{enumerate}[label=(C\arabic*), resume=cons]
      \item For any substitution $\sigma$, any $\Sigma\subseteq\Fm$, and any formula $\varphi$, if $\Sigma\vdash\varphi$, then $\sigma[\Sigma]\vdash\sigma(\varphi)$. (\emph{Structurality})
  \end{enumerate}
\end{dfn}

\begin{dfn}
  Given a logical language $\mathcal{L}$ and a set of variables $V$, a \emph{logic} or \emph{deductive system} in the language $\mathcal{L}$ is a pair $\mathcal{S}=\langle\Fmbf,\vdash_\mathcal{S}\rangle$, where $\Fmbf$ is the algebra of formulas of type $\mathcal{L}$ over $V$ and $\vdash_\mathcal{S}$ is a substitution-invariant consequence relation on $\Fm$, that is, a relation $\vdash_\mathcal{S}\,\subseteq\pow(\Fm)\times\Fm$ satisfying the conditions (C1), (C2), and (C4), and hence also (C3). Such a consequence relation is also called a syntactic consequence relation. (We will omit the subscript on $\vdash$ when there is no chance of confusion.) 
\end{dfn}

Lastly, the following is a variant of one of Tarski's original conditions.
\begin{dfn}\label{dfn:compact}
  A logic $\mathcal{S}=\langle\Fmbf,\vdash\rangle$ is said to be \emph{finitary} when its consequence relation satisfies the following additional property.
  \begin{enumerate}[label=(C\arabic*),resume=cons]
      \item For every $\Sigma\cup\{\varphi\}\subseteq\Fm$, if $\Sigma\vdash\varphi$, then there exists a finite $\Sigma^\prime\subseteq\Sigma$ such that $\Sigma^\prime\vdash\varphi$.
  \end{enumerate}
  The above property is also known as \emph{compactness}.
\end{dfn}

\begin{dfn}
  A \emph{(logical) matrix}\footnote{This concept was introduced by \L ukasiewicz and Tarski in the 1920's, an English translation of the work detailing this, titled ``Investigations into the sentential calculus'' can be found in \cite{Lukasiewicz1970}.} for a logical language $\mathcal{L}$, or an \emph{$\mathcal{L}$-matrix}, is an ordered pair $\langle\mathbf{A},F\rangle$ where $\mathbf{A}$ is an algebra of type $\mathcal{L}$ with universe $A$, and $F\subseteq A$; the algebra $\mathbf{A}$ is called the \emph{algebraic reduct} of the matrix and the set $F$ is called the set of \emph{designated values} or the \emph{filter} of the matrix.
\end{dfn}

A \emph{trivial} algebra is one with a one-element universe. All trivial algebras of the same type are isomorphic. In this paper, we will denote any trivial algebra of type $\mathcal{L}$ by $\mathbf{1}_\mathcal{L}$ and its universe by $\{\omega\}_\mathcal{L}$; we will drop the subscripts when the type is clear from the context. A \emph{trivial} $\mathcal{L}$-matrix is then the matrix $\left\langle\mathbf{1},\{\omega\}\right\rangle$, where $\mathbf{1}$ is the trivial algebra of type $\mathcal{L}$.

Given a logical language $\mathcal{L}$, and an $\mathcal{L}$-matrix $\langle\mathbf{A},F\rangle$, each formula $\varphi$ of type $\mathcal{L}$ over a set of variables $V$ has a unique interpretation in $\mathbf{A}$ depending on the values in $A$ (the universe of $\mathbf{A}$) that are assigned to its variables. Then, using the facts that $\Fmbf$ is absolutely freely generated from the set of variables and that $\mathbf{A}$ is an algebra over the same language, the interpretation of $\varphi$ can be expressed algebraically as $v(\varphi)$, where $v:\Fmbf\to\mathbf{A}$ is a homomorphism that maps each variable of $\varphi$ to its assigned value in $A$. Such a homomorphism whose domain is the formula algebra is called an \emph{assignment} or \emph{valuation}.

\begin{dfn}
  Given a logic $\mathcal{S}=\langle\Fmbf,\vdash\rangle$ in a language $\mathcal{L}$, an $\mathcal{L}$-matrix $\langle\mathbf{A},F\rangle$ is called a \emph{model of $\mathcal{S}$} if, for every valuation $v:\Fmbf\to\mathbf{A}$ and every $\Sigma\cup\{\varphi\}\subseteq\Fm$,
  \[
  \hbox{if }v[\Sigma]\subseteq F\hbox{ and }\Sigma\vdash\varphi\hbox{ then }v(\varphi)\in F.
  \]
  
  The class of all matrix models of $\mathcal{S}$ is denoted by $\mathsf{Mod}\,\mathcal{S}$.
\end{dfn}

\begin{dfn}
A logic $\mathcal{S}=\langle\Fmbf,\vdash\rangle$ in the language $\mathcal{L}$ is said to be \emph{complete relative to a class of $\mathcal{L}$-matrices} $\mathsf{M}$ if the following conditions are satisfied.
\begin{enumerate}[label=(\roman*)]
    \item $\mathsf{M}\subseteq\mathsf{Mod}\,\mathcal{S}$, i.e. each matrix in $\mathsf{M}$ is a model of $\mathcal{S}$;
    \item for every $\Sigma\cup\{\varphi\}\subseteq\Fm$ such that $\Sigma\not\vdash\varphi$, there is a matrix $\langle\mathbf{A},F\rangle\in\mathsf{M}$ and a valuation $v:\Fmbf\to\mathbf{A}$ such that $v[\Sigma]\subseteq F$ but $v(\varphi)\notin F$.
\end{enumerate}

Such a class of matrix models $\mathsf{M}$ is called a \emph{matrix semantics for} $\mathcal{S}$.
\end{dfn}

\begin{rem}\label{rem:mat_cons}
Logics may be defined using logical matrices as well. Let $\mathsf{M}$ be a class of $\mathcal{L}$-matrices. Then we can define a logic $\mathcal{S}_\mathsf{M}=\langle\Fmbf,\vdash_\mathsf{M}\rangle$, with $\vdash_\mathsf{M}$ defined as follows. For any $\Sigma\cup\{\varphi\}\subseteq\Fm$,
\[
\Sigma\vdash_\mathsf{M}\varphi\,\hbox{ iff for all }\,\langle\mathbf{A},F\rangle\in\mathsf{M}\,\hbox{ and for all valuations }\,v:\Fmbf\to\mathbf{A},\,v[\Sigma]\subseteq F\,\hbox{ implies }\,v(\varphi)\in F.
\]
It can then be observed that a logic $\mathcal{S}=\langle\Fmbf,\vdash\rangle$ in the language $\mathcal{L}$ is complete relative to a class of $\mathcal{L}$-matrices $\mathsf{M}$ when $\vdash$ coincides with $\vdash_\mathsf{M}$.
\end{rem}

\subsection{Logics of left variable inclusion and P\l{}onka sums}

The logics of left variable inclusion, or more precisely, the left variable inclusion companions of logics  and their connections with P\l{}onka sums of matrices have been discussed in \cite{Bonzio_2020}. We adapt parts of the discussion there as follows.

Suppose $\Fmbf$ is a formula algebra of type $\mathcal{L}$ over a set of variables $V$, with universe $\Fm$. For any formula $\varphi\in\Fm$, we denote the set of variables occurring in $\varphi$ by $\var(\varphi)$. Extending this notation, given any $\Gamma\subseteq\Fm$, we then set
\[
\var(\Gamma)=\bigcup\{\var(\varphi)\mid\,\varphi\in\Gamma\}.
\]

\begin{dfn}\label{dfn:left_var_inc}
Let $\mathcal{S}=\langle\Fmbf,\vdash\rangle$ be a logic. The \emph{left variable inclusion companion} of $\mathcal{S}$ is the pair $\mathcal{S}^l=\langle\Fmbf,\lvari\rangle$, where $\lvari\subseteq\pow(\Fm)\times\Fm$ is defined as follows.
\[
\Gamma\lvari\varphi\;\hbox{ iff }\;\hbox{there is a }\Gamma^\prime\subseteq\Gamma\hbox{ such that }\var(\Gamma^\prime)\subseteq\var(\varphi)\hbox{ and }\Gamma^\prime\vdash\varphi,
\]
where $\Gamma\cup\{\varphi\}\subseteq\Fm$.
\end{dfn}

\begin{rem}
Clearly, $\lvari\,\subseteq\,\vdash$. It can be easily checked that $\lvari$ as defined above satisfies the conditions (C1) -- (C4) in Definitions \ref{dfn:consrel}, \ref{dfn:structurality}, and is thus a substitution-invariant consequence relation. Hence the left variable inclusion companion $\mathcal{S}^l$ of the logic $\mathcal{S}$ is also a logic. It can also be checked that if $\vdash$ satisfies the condition (C5) in Definition \ref{dfn:compact}, then so does $\lvari$. Thus the left variable inclusion companions of finitary logics are finitary.
\end{rem}

It is known that the left variable inclusion companion of classical propositional logic is PWK \cite{Bonzio2017, CiuniCarrara2016, CiuniCarrara2019}. The left variable inclusion companions of Strong Kleene logic and of the paraconsistent logic LP have been considered in \cite{Szmuc2016}.

We next turn to P\l{}onka sums. These were introduced in \cite{Plonka1967} as a way of combining algebras in such a way that some of the properties of the original algebras are retained. The idea has found applications in various places (see the discussions on this in \cite{Bonzio2017, Bonzio_2020}).

\begin{dfn}
A \emph{directed system of algebras} consists of the following.

\begin{enumerate}[label=(\roman*)]
    \item A join semilattice $\mathbf{I}=\langle I,\le\rangle$;
    \item a family of algebras of the same type $\{\mathbf{A}_i\mid\,i\in I\}$ with disjoint universes;
    \item a homomorphism $f_{ij}:\mathbf{A}_i\to\mathbf{A}_j$, for every $i,j\in I$ with $i\le j$. The resulting set of homomorphisms must satisfy the following two conditions.
    \begin{enumerate}
        \item $f_{ii}$ is the identity homomorphism for each $i\in I$;
        \item for $i,j,k\in I$ such that $i\le j\le k$, $f_{ik}=f_{jk}\circ f_{ij}$.
    \end{enumerate}
\end{enumerate}
\end{dfn}

\begin{dfn}
Let $X$ be a directed system of algebras as defined above. The \emph{P\l{}onka sum} of $X$, denoted by $\Pl(X)$ or $\Pl(A_i)_{i\in I}$, is the algebra defined as follows.
\begin{enumerate}[label=(\roman*)]
    \item The universe of $\Pl(X)$ is the union $\displaystyle\bigcup_{i\in I}A_i$.
    \item The basic operations of $\Pl(X)$ are defined in terms of the homomorphisms $f_{ij}$ and the basic operations of the member algebras in $X$ as follows. For every basic operation $f$ of arity $n\ge1$, and $a_1,\ldots,a_n\in\displaystyle\bigcup_{i\in I}A_i$
    \[
    f^{\Pl(X)}(a_1,\ldots,a_n):=f^{A_j}\left(f_{i_1j}(a_1),\ldots,f_{i_nj}(a_n)\right),
    \]
    where $a_1\in A_{i_1},\ldots,a_n\in A_{i_n}$ and $j=i_1\lor\cdots\lor i_n$.
    
    In case nullary operations or constants\footnote{P\l onka sums of algebras with nullary operations were defined in \cite{Plonka1984}.} are present in the algebras in $X$, we need to assume that the indexing semilattice $\mathbf{I}$ has a bottom element, $\bot$, and for each nullary operation $f$, $f^{\Pl(X)}:=f^{A_\bot}$.
\end{enumerate}
\end{dfn}

The concept of directed systems of algebras is next extended to logical matrices in \cite{Bonzio_2020} as follows.

\begin{dfn}
A \emph{directed system of matrices} consists of the following.
\begin{enumerate}[label=(\roman*)]
    \item A join semilattice $\mathbf{I}=\langle I,\le\rangle$;
    \item a family of matrices $\{\langle\mathbf{A}_i,F_i\rangle\mid\,i\in I\}$, where $\mathbf{A}_i$ are algebras of the same type with disjoint universes $A_i$;
    \item a homomorphism $f_{ij}:\mathbf{A}_i\to\mathbf{A}_j$ such that $f_{ij}[F_i]\subseteq F_j$, for every $i,j\in I$ such that $i\le j$. The resulting set of homomorphisms must satisfy the following two conditions.
    \begin{enumerate}
        \item $f_{ii}$ is the identity homomorphism for each $i\in I$;
        \item for $i,j,k\in I$ such that $i\le j\le k$, $f_{ik}=f_{jk}\circ f_{ij}$.
    \end{enumerate}
\end{enumerate}
\end{dfn}

\begin{dfn}
Suppose $X$ is a directed system of matrices as defined above. The \emph{P\l{}onka sum} of $X$ is then defined as the following matrix.
\[
\Pl(X):=\left\langle\Pl(\mathbf{A}_i)_{i\in I},\displaystyle\bigcup_{i\in I}F_i\right\rangle.
\]
\end{dfn}

Given a class $\mathsf{M}$ of matrices, $\Pl(\mathsf{M})$ will denote the class of P\l{}onka sums of directed systems of matrices in $\mathsf{M}$.

Following \cite{Bonzio_2020}, we now turn to the construction of a special case of P\l{}onka sums of algebras. 

Suppose $\mathbf{A}$ is an algebra of some type $\mathcal{L}$ and $\mathbf{1}$ is the trivial algebra of the same type. Then the two-element family $\{\mathbf{A},\mathbf{1}\}$ equipped with the identity endomorphisms and the unique homomorphism $f:\mathbf{A}\to\mathbf{1}$ is a directed system of algebras. The P\l{}onka sum of this directed system of algebras is denoted by $\mathbf{A}\oplus\mathbf{1}$. Thus $\mathbf{A}\oplus\mathbf{1}$ is the $\mathcal{L}$-algebra with universe $A\cup\{\omega\}$. The basic operations of this algebra are defined as follows. For any $n$-ary basic operation $f$,
\[
f^{\mathbf{A}\oplus\mathbf{1}}(a_1,\ldots,a_n):=\left\{\begin{array}{ll}
     f^{\mathbf{A}}(a_1,\ldots,a_n)&\hbox{if }a_1,\ldots,a_n\in A\\
     \omega&\hbox{otherwise}. 
\end{array}\right.
\]

This construction can then be lifted to matrices as follows. Suppose $\langle\mathbf{A},F\rangle$ is an $\mathcal{L}$-matrix and $\langle\mathbf{1},\{\omega\}\rangle$ is the trivial $\mathcal{L}$-matrix. Then the two-element family $\left\{\langle\mathbf{A},F\rangle,\langle\mathbf{1},\{\omega\}\rangle\right\}$ equipped with the identity endomorphisms and the unique homomorphism $f:\mathbf{A}\to\mathbf{1}$ is a directed system of matrices. The P\l{}onka sum of this directed system of matrices is the $\mathcal{L}$-matrix $\left\langle\mathbf{A}\oplus\mathbf{1},F\cup\{\omega\}\right\rangle$.

Lastly, the following two theorems proved in \cite{Bonzio_2020} are included below.

\begin{thm}{\cite[Lemma 13]{Bonzio_2020}}\label{thm:Plonka_sound}
Let $\mathcal{S}=\langle\Fmbf,\vdash\rangle$ be a logic and $X$ be a directed system of models of $\mathcal{S}$. Then $\Pl(X)$ is a model of the left variable inclusion companion, $\mathcal{S}^l=\langle\Fmbf,\lvari\rangle$, of $\mathcal{S}$.
\end{thm}
  
\begin{thm}{\cite[Theorem 14]{Bonzio_2020}}\label{thm:Plonka_complete}
  Let $\mathcal{S}=\langle\Fmbf,\vdash\rangle$ be a logic and $\mathsf{M}$ be a class of matrices, of the same type as $\Fmbf$, containing the trivial matrix $\langle\mathbf{1},\{\omega\}\rangle$. If $\mathcal{S}$ is complete relative to $\mathsf{M}$, then $\mathcal{S}^l=\langle\Fmbf,\lvari\rangle$ is complete relative to $\Pl(\mathsf{M})$.
\end{thm}

\section{Left variable inclusion and restricted rules companions of logics and paraconsistency}\label{sec:lvari}

In the following discussion, $\mathcal{L}$ is a logical language, $\Fmbf$ is the formula algebra of type $\mathcal{L}$ over a countably infinite set of variables $V$, and $\Fm$ is the universe of $\Fmbf$.

The next two lemmas follow immediately from the definition of a logic of left variable inclusion.

\begin{lem}\label{lem:S-thm_iff_Sl-thm}
  Suppose $\mathcal{S}=\langle\Fmbf,\vdash\rangle$ is a logic and $\mathcal{S}^l=\langle\Fmbf,\lvari\rangle$ is the left variable inclusion companion of $\mathcal{S}$. Then for any formula $\varphi\in\Fm$, $\vdash\varphi$ iff $\lvari\varphi$.
\end{lem}

\begin{lem}\label{lem:Sl_cons}
Suppose $\mathcal{S}=\langle\Fmbf,\vdash\rangle$ is a logic and $\mathcal{S}^l=\langle\Fmbf,\lvari\rangle$ is the left variable inclusion companion of $\mathcal{S}$. Then for any $\Sigma\cup\{\varphi\}\subseteq\Fm$, if $\var(\varphi)\cap\var(\Sigma)=\emptyset$ and $\not\vdash\varphi$, then $\Sigma\not\lvari\varphi$.
\end{lem}

We recall here that a logic $\mathcal{S}$ may be induced syntactically via a given Hilbert-style presentation\footnote{For definitions of Hilbert-style presentations, axioms, rules of inference, and how a logic may be induced via this method, one can see, for example, \cite{FontJansanaPigozzi2003,Mendelson}}. Henceforth, such a logic induced by a given Hilbert-style presentation will be referred to as a \emph{Hilbert-style logic}.

\begin{dfn}\label{dfn:res_rules_comp}
Suppose $\mathcal{S}=\langle\Fmbf,\vdash\rangle$ is a Hilbert-style logic with $A\subseteq\Fm$ as the set of axioms and $R_\mathcal{S}\subseteq\pow(\Fm)\times\Fm$ as the set of rules of inference.

We now define the \emph{restricted rules companion} of $\mathcal{S}$, and denote it by $\mathcal{S}^{re}=\langle\Fmbf,\lrri\rangle$, as the Hilbert-style logic with the following sets of axioms and rules.

\begin{tabular}{lcl}
       Set of axioms&=&$A$, and\\
       set of rules of inference&=&$R_{\mathcal{S}^{re}}=\left\{\dfrac{\Gamma}{\alpha}\in R_\mathcal{S}\mid\,\var(\Gamma)\subseteq\var(\alpha)\right\}$.
  \end{tabular}
\end{dfn}
  
For example, in \cite{Bonzio2017}, we see PWK as an example of the restricted rules companion of classical propositional logic (CPC).

\begin{thm}\label{thm:S-thm_iff_Sr-thm}
  Suppose $\mathcal{S}=\langle\Fmbf,\vdash\rangle$ is a Hilbert-style logic with $A$ and $R_\mathcal{S}$ as its sets of axioms and rules of inference, respectively. Let $\mathcal{S}^{re}=\langle\Fmbf,\lrri\rangle$ be the restricted rules companion of $\mathcal{S}$ as described above. Then for any $\varphi\in\Fm$, $\lrri\varphi$ iff $\vdash\varphi$.
\end{thm}

\begin{proof}
    The left-to-right direction is easy to see, because if $\lrri\varphi$, then there is a proof $D$ of $\varphi$ that consists of instances of the axioms in $A$ and rules in $R_{\mathcal{S}^{re}}$. Since members of $A$ are also axioms for $\mathcal{S}$, and each rule in $R_{\mathcal{S}^{re}}$ is an instance of the unrestricted version of it in $R_\mathcal{S}$, $D$ is also a proof of $\varphi$ in the logic $\mathcal{S}$.
  
  For the right-to-left direction, suppose $\vdash\varphi$ and $D=\langle\varphi_1,\ldots,\varphi_n\rangle$ is a proof of $\varphi$ in the deductive system for $\mathcal{S}$. We will show, by induction on the length $n$ of $D$, that this proof of $\varphi$ can be translated into another proof of it that uses only the rules in $R_{\mathcal{S}^{re}}$.
  
  For the base case, that is, when $n=1$, $\varphi=\varphi_1$ must be an axiom and there is nothing to prove. 
  
  As our induction hypothesis, we assume that the claim holds for all $1\le l<n$, where $n$ is some positive integer. That is, each $\varphi_l$ in $D$, where $1\le l<n$, has a proof in $\mathcal{S}^{re}$.
  
  For the induction step, we need to show that $\varphi_n$ has a proof in $\mathcal{S}^{re}$. If $\varphi_n$ is an axiom, then the argument is the same as in the base case. Now suppose $\varphi_n$ is obtained via one of the rules $\Theta\in R_{\mathcal{S}}$ from some $D^\prime\subseteq\{\varphi_1,\ldots,\varphi_{n-1}\}$. So by the induction hypothesis, there is a proof in $\mathcal{S}^{re}$ for each formula in $D^\prime$. Let $\langle\psi_1,\ldots,\psi_m\rangle$ be the result of gluing these proofs. Now consider a substitution of variables $\sigma$ defined by 
  \[
  \sigma(p)=\left\{
  \begin{array}{ll}
  p&\hbox{if }p\in\var(\varphi_n)\\
  a&\hbox{otherwise}
  \end{array}\right.,
  \]
  where $a$ is some fixed variable in $\varphi_n$ or a 0-ary operator in the logical language if $\var(\varphi_n)=\emptyset$ (there must be at least one such operator if $\var(\varphi_n)=\emptyset$). 
  
  We note that $\sigma$ transforms any instance of an axiom into another instance of the same axiom. Moreover, if for some $t<m$, $\psi_t$ is obtained from $\psi_{t_1},\ldots,\psi_{t_k},1\le t_1,\ldots,t_k<t$ by an application of a rule $\Psi$ in $R_{\mathcal{S}^{re}}$, then $\var(\{\psi_{t_1},\ldots,\psi_{t_k}\})\subseteq\var(\psi_t)$. 
  
  This implies that $\var(\{\sigma(\psi_{t_1}),\ldots,\sigma(\psi_{t_k})\})\subseteq\var(\sigma(\psi_t))$. Hence $\sigma(\psi_t)$ can be obtained by the application of the same rule, $\Psi$, on $\sigma(\psi_{t_1}),\ldots,\sigma(\psi_{t_k})$. Thus $\langle\sigma(\psi_1),\ldots,\sigma(\psi_m)\rangle$ is still a proof in $\mathcal{S}^{re}$. Now $\sigma$ does not change the form of any $\psi_s$ in $\langle\psi_1,\ldots,\psi_m\rangle$, and $\var(\sigma(\psi_s))\subseteq\var(\varphi_n)$ for all $1\le s\le m$. Thus $\Theta$ can be applied on the formulas in $\sigma(D^\prime)$ to obtain $\varphi_n$ and $\langle\sigma(\psi_1),\ldots,\sigma(\psi_m),\varphi_n\rangle$ is a proof of $\varphi_n$ in $\mathcal{S}^{re}$.
  
  Hence by the principle of mathematical induction, any proof of a theorem $\varphi$ in $\mathcal{S}$ can be translated to a proof of it in $\mathcal{S}^{re}$.
\end{proof}

\begin{rem}
Thus, given a Hilbert-style logic $\mathcal{S}$, one can define two companion logics to it, namely, $\mathcal{S}^l$ and $\mathcal{S}^{re}$. Lemma \ref{lem:S-thm_iff_Sl-thm} and Theorem \ref{thm:S-thm_iff_Sr-thm} show that both these companion logics have the same theorems as the logic $\mathcal{S}$. The question that naturally arises here is the following. Are these two companion logics the same? We answer this question below in Theorem \ref{thm:Sr_sub_Sl} and Remark \ref{rem:Sr_neq_Sl}.
\end{rem}

\begin{thm}\label{thm:Sr_sub_Sl}
  Let $\mathcal{S}=\langle\Fmbf,\vdash\rangle$ be a Hilbert-style logic with $A$ and $R_\mathcal{S}$ as its sets of axioms and rules of inference, respectively, and $\mathcal{S}^l=\langle\Fmbf,\lvari\rangle$ be the left variable inclusion companion of $\mathcal{S}$. Let $\mathcal{S}^{re}=\langle\Fmbf,\lrri\rangle$ be the restricted rules companion of $\mathcal{S}$ as described in Definition \ref{dfn:res_rules_comp}.
  
  Then for any $\Sigma\cup\{\varphi\}\subseteq\Fm$, if $\Sigma\lrri\varphi$, then $\Sigma\lvari\varphi$, that is, $\lrri\,\subseteq\,\lvari$.
\end{thm}

\begin{proof}
  Suppose $\Sigma\cup\{\varphi\}\subseteq\Fm$ such that $\Sigma\lrri\varphi$.
  
  If $\var(\Sigma)\subseteq\var(\varphi)$, then clearly, $\Sigma\lvari\varphi$. This case includes the situation where $\Sigma=\emptyset$.
  
  Suppose now that $\Sigma\neq\emptyset$ and $\var(\Sigma)\not\subseteq\var(\varphi)$. Let $\Delta=\{\gamma\in\Sigma\mid\,\var(\gamma)\subseteq\var(\varphi)\}$. We first show that $\Delta\lrri\varphi$.
  
  \textbf{Case 1:} $\varphi$ is a theorem, i.e, $\varphi$ is derived using only the axioms in $A$ and the rules in $R_\mathcal{S}$.
  
  In this case, $\lrri\varphi$. Thus, $\Delta\lrri\varphi$.
  
  \textbf{Case 2:} $\varphi\in\Sigma$.
  
  Then $\varphi\in\Delta$ and hence $\Delta\lrri\varphi$.
  
  \textbf{Case 3:} Now suppose that $\varphi$ is neither a theorem nor an element of $\Sigma$ and that any derivation of $\varphi$ from $\Sigma$ contains some elements of $\Sigma$. Let $D$ be a derivation of $\varphi$ from $\Sigma$. We discard from $D$ any redundancy, i.e. any wff that is not required for obtaining $\varphi$. Let this resulting derivation be $D^\prime=\langle\alpha_1,\ldots,\alpha_n=\varphi\rangle$. Thus $D^\prime$ is a minimal derivation of $\varphi$ in the sense that $D^\prime$ cannot be shortened any further to another derivation of $\varphi$ from $\Sigma$.
  
  Now, for any $\alpha_i\in D^\prime$ that is used by any rule of inference in $R_{\mathcal{S}^{re}}$ to obtain $\alpha_j\in D^\prime$, $\var(\alpha_i)\subseteq\var(\alpha_j)$. If $\alpha_j$ is in turn used by a rule of inference to obtain $\alpha_k\in D^\prime$, then $\var(\alpha_j)\subseteq\var(\alpha_k)$, which implies that $\var(\alpha_i)\subseteq\var(\alpha_k)$.
  
 Since every formula in $D^\prime$ contributes to the derivation of $\varphi$ and $\varphi$ must have been obtained from some formulas in $D^\prime\setminus\{\varphi\}$, $\var(D^\prime)\subseteq\var(\varphi)$. Thus, in particular, any element of $\Sigma$ in $D^\prime$ must be in $\Delta$. So $D^\prime$ is also a derivation of $\varphi$ from $\Delta$. Hence $\Delta\lrri\varphi$.
 
 Thus in all cases, we have $\Delta\lrri\varphi$. Now, the Hilbert-style presentations of $\mathcal{S}$ and $\mathcal{S}^{re}$ have the same set of axioms and any rule in $R_{\mathcal{S}^{re}}$ is an instance of its unrestricted version in $R_\mathcal{S}$. So $\Delta\vdash\varphi$.
 
 Thus we have a subset $\Delta$ of $\Sigma$ with $\var(\Delta)\subseteq\var(\varphi)$ such that $\Delta\vdash\varphi$. Hence $\Sigma\lvari\varphi$.
\end{proof}

\begin{rem}\label{rem:Sr_neq_Sl}
It is, however, not true in general, that for a Hilbert-style logic $\mathcal{S}=\langle\Fmbf,\vdash\rangle$, $\lvari\,\subseteq\,\lrri$, where $\mathcal{S}^l=\langle\Fmbf,\lvari\rangle$ and $\mathcal{S}^{re}=\langle\Fmbf,\lrri\rangle$ are the left variable inclusion and the restricted rules companions of $\mathcal{S}$, respectively. This can be seen from the following minimal example.
  
Suppose $\mathcal{S}=\langle\Fmbf,\vdash\rangle$ is a Hilbert-style logic over the language $\mathcal{L}=\{\land,\lor\}$ with an empty set of axioms and the following two rules of inference.
\[
R_1:\,\dfrac{\alpha\land\beta}{\alpha}\quad\hbox{and}\quad R_2:\,\dfrac{\alpha}{\alpha\lor\beta},\quad\hbox{where }\alpha,\beta\in\Fm
\]
Let $\mathcal{S}^l=\langle\Fmbf,\lvari\rangle$ be the left variable inclusion companion of $\mathcal{S}$. The restricted rules companion of $\mathcal{S}$, $\mathcal{S}^{re}=\langle\Fmbf,\lrri\rangle$, is then the logic induced by the same set of axioms and the following two rules of inference.
\[
\begin{array}{ll}
     R_1^\prime:\,\dfrac{\alpha\land\beta}{\alpha}&\hbox{such that }\var(\alpha\land\beta)\subseteq\var(\alpha),\hbox{ i.e. }\var(\beta)\subseteq\var(\alpha),\hbox{ and}\\
     &\\
     R_2:\,\dfrac{\alpha}{\alpha\lor\beta}.&
\end{array}
\]
($R_2$ does not need to be restricted as  $\var(\alpha)\subseteq\var(\alpha\lor\beta)$ for any $\alpha,\beta\in\Fm$.)

Now, suppose $p,q$ are distinct variables. We consider the following derivation in $\mathcal{S}$.
\[
\begin{array}{lcl}
     p\land q&\vdash&1.\, p\land q\\
     &&2.\, p\qquad [R_1\hbox{ on (1)}]\\
     &&3.\, p\lor q\qquad [R_2\hbox{ on (2)}]
\end{array}
\]
Thus $p\land q\vdash p\lor q$. Since $\var(p\land q)=\var(p\lor q)$, $p\land q\lvari p\lor q$. However, it may be noted that we do not have $p\land q\lvari p$.

On the other hand, $p\land q\not\lrri p\lor q$ since we cannot replace the application of the rule $R_1$ in step 2 of the above derivation by an application of $R_1^\prime$.

Thus $\lvari\,\not\subseteq\,\lrri$, i.e. in this case, $\lrri\,\subsetneq\,\lvari$, and hence $\mathcal{S}^l$ is different from $\mathcal{S}^{re}$.

In fact, the $\land$-$\lor$ fragment of classical propositional logic (the logic of distributive lattices) can also be used in place of the logic $\mathcal{S}$ described above. The minimal example, however, in enough to show that the two companions of a logic can differ, and also how that might happen. We have given more examples to illustrate this inequality between the two companion logics later in the paper (see Remarks \ref{rem:HPRL_SrXSl} and \ref{rem:RML_SrXSl}).
\end{rem}

\begin{rem}
One can justify the above observation in the previous remark by noting that the restricted rules companion of a logic enforces a variable inclusion restriction ``locally'' at each step of a derivation. On the other hand, the left variable inclusion companion of a logic outsources the steps of a derivation to the original logic with the unrestricted rules of inference and only enforces the variable inclusion restriction ``globally'' on the overall entailment. 
\end{rem}

Now, for the following theorem and the subsequent discussions we assume that the logical language $\mathcal{L}$ contains a unary operation $\neg$, that is intended to denote negation. $V,\Fmbf,\Fm$ are as before.

\begin{thm}\label{thm:Sl-paraconsistent}
  Suppose $\mathcal{S}=\langle\Fmbf,\vdash\rangle$ is a logic and that there exists an $\alpha\in\Fm$ such that $\not\vdash\alpha$, and a $p\in V$ such that $p\notin\var(\alpha)$. Then ECQ fails in $\mathcal{S}^l=\langle\Fmbf,\lvari\rangle$, the left variable inclusion companion of $\mathcal{S}$.
\end{thm}

\begin{proof}
  It follows from Lemma \ref{lem:Sl_cons} that $\{p,\neg p\}\not\lvari\alpha$, since $\{p\}\cap\var(\alpha)=\emptyset$.
\end{proof}

\begin{rem}
For other deductive failures in logics of left variable inclusion, such as $p\land q\not\lvari p$, where $p,q$ are distinct variables, one can see the discussion in \cite{CiuniCarrara2016} for the case of PWK, the left variable inclusion companion of classical propositional logic.
\end{rem}

\begin{cor}\label{cor:Sr-paraconsistent}
Suppose $\mathcal{S}=\langle\Fmbf,\vdash\rangle$ is a Hilbert-style logic and that there exists an $\alpha\in\Fm$ such that $\not\vdash\alpha$, and a $p\in V$ such that $p\notin\var(\alpha)$. Then ECQ fails in $\mathcal{S}^{re}=\langle\Fmbf,\lrri\rangle$, the restricted rules companion of $\mathcal{S}$.
\end{cor}

\begin{proof}
  This follows from the fact that $\lrri\,\subseteq\,\lvari$ proved in Theorem \ref{thm:Sr_sub_Sl}.
\end{proof}

\begin{rem}
The above theorem shows that the left variable inclusion companion of any non-trivial logic with a sufficient supply of variables, is, at least, weakly paraconsistent. This failure of ECQ is regardless of whether the original logic is explosive or not.

We also note that due to Lemma \ref{lem:S-thm_iff_Sl-thm}, LNC holds in the logic $\mathcal{S}$ iff it holds in $\mathcal{S}^l$. Now, suppose $\mathcal{S}$ is a logic where ECQ holds but LNC fails. Then as discussed in Section \ref{sec:introduction}, $\mathcal{S}$ is not paraconsistent (an example of such a logic is $L_3$). However, the left variable inclusion companion of $\mathcal{S}$ will be both non-explosive and without LNC due to Lemma \ref{lem:S-thm_iff_Sl-thm} and Theorem  \ref{thm:Sl-paraconsistent}, i.e. ECQ and LNC will both fail in $\mathcal{S}^l$. 

Similar remarks can be made about the restricted rules companion of a Hilbert-style logic. We thus have the following theorem.
\end{rem}

\begin{thm}
  Suppose $\mathcal{S}$ is a logic where LNC fails. Then $\mathcal{S}^l$ is strongly paraconsistent. Moreover, if $\mathcal{S}$ is a Hilbert-style logic, then the logic $\mathcal{S}^{re}$ will also be strongly paraconsistent.
\end{thm}

\begin{proof}
  ECQ fails in $\mathcal{S}^l$ by Theorem \ref{thm:Sl-paraconsistent} and in $\mathcal{S}^{re}$ by Corollary \ref{cor:Sr-paraconsistent}. Since LNC fails in $\mathcal{S}$, it fails in $\mathcal{S}^l$ and $\mathcal{S}^{re}$ as well, by the above remark. Thus $\mathcal{S}^l$ and $\mathcal{S}^{re}$ are strongly paraconsistent.
\end{proof}

\begin{rem}\label{rem:Sl&Sr_paracons}
We can summarize the above results as follows.
\begin{center}
    \begin{tabular}{|c|c|}
         \hline
         $\mathcal{S}$&$\mathcal{S}^l,\,\mathcal{S}^{re}$\\
         \hline
         Not paraconsistent, LNC holds&Weakly paraconsistent\\
         Not paraconsistent, LNC fails&Strongly paraconsistent\\
         Weakly paraconsistent&Weakly paraconsistent\\
         Strongly paraconsistent&Strongly paraconsistent\\
         \hline
    \end{tabular}
\end{center}
\end{rem}

\section{Deduction theorem and a sufficient condition for \texorpdfstring{$\lvari\,=\,\lrri$}{LeftVarIncl=ResRulComp}}\label{sec:suffcond-Sr=Sl}

The first part of this section is concerned with the Deduction theorem and its converse in left variable inclusion companion logics. These are important meta-theorems that can be proved for classical, intuitionistic and many other non-classical logics, including some paraconsistent logics. We assume that the logical language $\mathcal{L}$ contains a binary operation $\limp$, that is intended to denote implication. $V,\Fmbf,\Fm$ are as before. Then, given a logic $\mathcal{S}=\langle\Fmbf,\vdash\rangle$, the Deduction theorem and its converse (if they hold in $\mathcal{S}$) together assert that, for any $\Sigma\cup\{\alpha,\beta\}\subseteq\Fm$, $\Sigma\cup\{\alpha\}\vdash\beta$ iff $\Sigma\vdash\alpha\limp\beta$.

\begin{thm}\label{thm:Ded_Sl}
  Suppose $\mathcal{S}=\langle\Fmbf,\vdash\rangle$ is a logic where the Deduction theorem holds. Then the Deduction theorem also holds in its left variable inclusion companion $\mathcal{S}^l=\langle\Fmbf,\lvari\rangle$, i.e. for any $\Sigma\cup\{\alpha,\beta\}\subseteq\Fm$, if $\Sigma\cup\{\alpha\}\lvari\beta$ then $\Sigma\lvari\alpha\limp\beta$.
\end{thm}

\begin{proof}
  Suppose $\Sigma\cup\{\alpha,\beta\}\subseteq\Fm$ and $\Sigma\cup\{\alpha\}\lvari\beta$.
  
  Then there exists a $\Delta\subseteq\Sigma\cup\{\alpha\}$ such that $\var(\Delta)\subseteq\var(\beta)$ and $\Delta\vdash\beta$. The two possible cases that arise from here are as follows.
  
  \textbf{Case 1}: $\alpha\notin\Delta$
  
  In this case, $\Delta\subseteq\Sigma$. Now since $\Delta\vdash\beta$, by monotonicity, $\Delta\cup\{\alpha\}\vdash\beta$. Then by the Deduction theorem in $\mathcal{S}$, $\Delta\vdash\alpha\limp\beta$.
  
  Finally, since $\var(\Delta)\subseteq\var(\beta)$, $\var(\Delta)\subseteq\var(\alpha\limp\beta)=\var(\alpha)\cup\var(\beta)$.
  
  Thus $\Delta\subseteq\Sigma$ such that $\var(\Delta)\subseteq\var(\alpha\limp\beta)$ and $\Delta\vdash\alpha\limp\beta$. Hence $\Sigma\lvari\alpha\limp\beta$.
  
  \textbf{Case 2}: $\alpha\in\Delta$
  
  In this case, there exists $\Delta^\prime\subseteq\Sigma$ such that $\Delta=\Delta^\prime\cup\{\alpha\}$. 
  
  Also, $\var(\Delta^\prime)\subseteq\var(\Delta)\subseteq\var(\beta)\subseteq\var(\alpha\limp\beta)=\var(\alpha)\cup\var(\beta)$.
  
  Now, $\Delta=\Delta^\prime\cup\{\alpha\}\vdash\beta$ implies $\Delta^\prime\vdash\alpha\limp\beta$ by the Deduction theorem in $\mathcal{S}$.
  
  Thus we have $\Delta^\prime\subseteq\Sigma$ such that $\var(\Delta^\prime)\subseteq\var(\alpha\limp\beta)$ and $\Delta^\prime\vdash\alpha\limp\beta$. Hence $\Sigma\lvari\alpha\limp\beta$.
\end{proof}

Although the Deduction theorem passes through to the left variable inclusion companion unscathed, the same is not the case for its converse as shown in the following theorem.

\begin{thm}\label{thm:Det_Sl}
  Suppose $\mathcal{S}=\langle\Fmbf,\vdash\rangle$ is a logic where the converse of the Deduction theorem holds. Then the following restricted version of the converse of the Deduction theorem holds in $\mathcal{S}^l=\langle\Fmbf,\lvari\rangle$. For any $\Sigma\cup\{\alpha,\beta\}\subseteq\Fm$, $\Sigma\lvari\alpha\limp\beta$ implies $\Sigma\cup\{\alpha\}\lvari\beta$ iff $\var(\alpha)\subseteq\var(\beta)$.
\end{thm}

\begin{proof}
  Suppose $\Sigma\cup\{\alpha,\beta\}\subseteq\Fm$ such that $\Sigma\lvari\alpha\limp\beta$ and $\var(\alpha)\subseteq\var(\beta)$. 
  Then there exists a $\Delta\subseteq\Sigma$ such that $\var(\Delta)\subseteq\var(\alpha\limp\beta)=\var(\alpha)\cup\var(\beta)$ and $\Delta\vdash\alpha\limp\beta$.
  
  So by the converse of the Deduction theorem in $\mathcal{S}$, we have $\Delta\cup\{\alpha\}\vdash\beta$. Also, $\Delta\cup\{\alpha\}\subseteq\Sigma\cup\{\alpha\}$.
  
  Thus we have $\Delta\cup\{\alpha\}\subseteq\Sigma\cup\{\alpha\}$ such that $\Delta\cup\{\alpha\}\vdash\beta$. So $\Sigma\cup\{\alpha\}\lvari\beta$ iff $\var(\Delta\cup\{\alpha\})\subseteq\var(\beta)$.
  
  Now, given that $\var(\Delta)\subseteq\var(\alpha)\cup\var(\beta)$, it follows immediately from $\var(\alpha)\subseteq\var(\beta)$ that $\var(\Delta\cup\{\alpha\})\subseteq\var(\beta)$. Conversely, if $\var(\Delta\cup\{\alpha\})\subseteq\var(\beta)$, then $\var(\alpha)\subseteq\var(\beta)$.
  
  Hence $\Sigma\lvari\alpha\limp\beta$ implies $\Sigma\cup\{\alpha\}\lvari\beta$ iff $\var(\alpha)\subseteq\var(\beta)$.
\end{proof}

We have seen earlier, in Remark \ref{rem:Sr_neq_Sl}, that the restricted rules companion of a Hilbert-style logic does not always coincide with its left variable inclusion companion, but is always contained in it (Theorem \ref{thm:Sr_sub_Sl}). On the other hand, PWK is an example of a logic that is both the restricted rules companion and the left variable inclusion companion of CPC. Thus an interesting question to investigate is as follows. What are the necessary and sufficient conditions for the two companions of a logic to coincide? While a necessary condition is still under investigation, we present a sufficient condition in the next theorem.

\begin{thm}\label{thm:deduction_thm-Sr=Sl}
  Suppose $\mathcal{S}=\langle\Fmbf,\vdash\rangle$ is a finitary Hilbert-style logic such that $\dfrac{\alpha,\alpha\limp\beta}{\beta}$ (modus ponens [MP]) is a rule of inference in $\mathcal{S}$. Suppose further that the Deduction theorem holds in $\mathcal{S}$. Then the restricted rules companion of $\mathcal{S}$ coincides with the left variable inclusion companion of $\mathcal{S}$, i.e. $\mathcal{S}^{re}=\langle\Fmbf,\lrri\rangle=\langle\Fmbf,\lvari\rangle=\mathcal{S}^l$.
\end{thm}

\begin{proof}
  Let $\Sigma\cup\{\alpha\}\subseteq\Fm$ such that $\Sigma\lvari\alpha$. Then by Definition \ref{dfn:left_var_inc}, there exists $\Delta\subseteq\Sigma$ such that $\var(\Delta)\subseteq\var(\alpha)$ and $\Delta\vdash\alpha$. Now, since $\mathcal{S}$ is finitary, we can find such a $\Delta$ that is finite. Thus $\Delta$ can be assumed to be finite, without loss of generality. Let $\Delta=\{\varphi_1,\ldots,\varphi_n\}$. Then, using the Deduction theorem in $\mathcal{S}$ $n$ times, we have $\vdash(\varphi_1\limp(\ldots\limp(\varphi_n\limp\alpha)\ldots))$. So by Theorem \ref{thm:S-thm_iff_Sr-thm}, $\lrri(\varphi_1\limp(\ldots\limp(\varphi_n\limp\alpha)\ldots))$. 
  
  Now, since MP is a rule of inference in $\mathcal{S}$, the restricted version of it, that is, $\dfrac{\alpha,\alpha\limp\beta}{\beta}$, provided $\var(\alpha)\subseteq\var(\beta)$, is a rule of inference in $\mathcal{S}^{re}$. Since $\var(\Delta)\subseteq\var(\alpha)$, we have $\Delta=\{\varphi_1,\ldots,\varphi_n\}\lrri\alpha$,
  by applying the restricted modus ponens $n$ times. Thus $\Sigma\lrri\alpha$. Hence $\lvari\,\subseteq\,\lrri$. Since $\lrri\,\subseteq\,\lvari$, by Theorem \ref{thm:Sr_sub_Sl}, this implies that $\lrri\,=\,\lvari$, and thus $\mathcal{S}^{re}=\mathcal{S}^l$.
\end{proof}

\begin{rem}
The fact that $\mathrm{CPC}^l=\mathrm{PWK}=\mathrm{CPC}^{re}$ can now be seen as a corollary of the above theorem.
\end{rem}

We now turn to a couple more examples of left variable inclusion and restricted rules companions of logics. For one of the examples, the two companion logics coincide, while they are different in case of the other one.

\subsection{Intuitionistic paraconsistent weak Kleene logic (IPWK)}\label{subsec:ipwk}

Our first example deals with the left variable inclusion and restricted rules companion logics of intuitionistic propositional logic (IPC).

\begin{thm}\label{thm:IPWK=HIPWK}
  The left variable inclusion companion of IPC coincides with its restricted rules companion, that is, a Hilbert-style presentation for the left variable inclusion companion of IPC can be obtained by keeping the same set of axioms and restricting the rules of inference.
\end{thm}

\begin{proof}
  It is well known (one can see, for example, \cite[Volume I, Chapter 2]{VDT}) that MP is a rule of inference in a Hilbert-style presentation of IPC. The Deduction theorem for IPC can be proved in much the same way as in classical propositional logic (CPC).
  
  This theorem is then a corollary of Theorem \ref{thm:deduction_thm-Sr=Sl}.
\end{proof}

Alternatively, a direct proof of the above result is possible following the technique used in \cite{Bonzio2017} to show that PWK, the left variable inclusion companion of CPC, is also its restricted rules companion. This is detailed below.

Let $\mathcal{L}=\{\land,\lor,\limp,\neg,0,1\}$ be a logical language, where the arities of the operators $\land,\lor,\limp,\neg,0,1$ are $2,2,2,1,0$, and 0, respectively. Suppose $V$ is a countable set of propositional variables and let $\Fmbf$ denote the formula algebra over $V$ of type $\mathcal{L}$. 

Now, any Heyting algebra can be seen as an algebra of the above type $\mathcal{L}$. It is well known that IPC is sound and complete with respect to valuations in Heyting algebras \cite[Volume II, Chapter 13, \S 5]{VDT}. IPC can be described syntactically as the logic $\left\langle\Fmbf,\inyields\right\rangle$, where $\Fmbf$ is the formula algebra of type $\mathcal{L}$ and $\inyields$ is a substitution-invariant (syntactic) consequence relation on $\Fm$. On the other hand, one can consider the class of $\mathcal{L}$-matrices:
\[
\mathcal{H}=\{\langle\mathbf{H},\{1\}\rangle\mid\,\mathbf{H}\hbox{ is a Heyting algebra}\}
\]
and describe IPC semantically as the logic $\left\langle\Fmbf,\inmodels\right\rangle$, where $\inmodels\,\subseteq\pow(\Fm)\times\Fm$ is defined as follows. For any $\Sigma\cup\{\alpha\}\subseteq\Fm$, 
\[
\Sigma\inmodels\alpha\hbox{ iff for all } \langle\mathbf{H},\{1\}\rangle\in\mathcal{H}\hbox{ and  for all valuations }v:\Fmbf\to\mathbf{H}, v[\Sigma]\subseteq\{1\}\hbox{ implies } v(\alpha)=1.
\]
Then by the soundness and completeness of IPC with respect to valuations in Heyting algebras, we have for any $\Sigma\cup\{\alpha\}\subseteq\Fm$, 
\[
\Sigma\inyields\alpha\hbox{ iff }\Sigma\inmodels\alpha,\hbox{ i.e. }\inyields\,=\,\inmodels.
\]
In other words, $\langle\Fmbf,\inyields\rangle$ is complete relative to the class of matrices $\mathcal{H}$.

We would like to point out that there is a slight deviation from the notation used in Section \ref{sec:aal}, the logic $\langle\Fmbf,\inmodels\rangle$ is defined using the class of matrices $\mathcal{H}$ following the recipe indicated in Remark \ref{rem:mat_cons}, and hence is the logic $\langle\Fmbf,\vdash_\mathcal{H}\rangle$. However, we feel that $\langle\Fmbf,\inmodels\rangle$ is a slightly more indicative nomenclature, at least in this case. We will follow this style for the other known logics to come in this paper.

Now, the trivial $\mathcal{L}$-matrix, $\langle\mathbf{1},\{\omega\}\rangle\in\mathcal{H}$. So by Theorem \ref{thm:Plonka_complete}, the left variable inclusion companion of IPC is complete relative to $\Pl(\mathcal{H})$, the class of P\l{}onka sums of the directed systems of matrices in $\mathcal{H}$.

In particular, for each Heyting algebra $\mathbf{H}$, the P\l{}onka sum of the matrices $\langle\mathbf{H},\{1\}\rangle$ and $\langle\mathbf{1},\{\omega\}\rangle$, $\langle\mathbf{H}\oplus\mathbf{1},\{1,\omega\}\rangle\in\Pl(\mathcal{H}\rangle$. For any Heyting algebra $\mathbf{H}$, we will refer to the algebra $\mathbf{H}\oplus\mathbf{1}$ as the \emph{extended Heyting algebra} corresponding to $\mathbf{H}$, and denote it by $\sharpE{H}$. We note that the operations in $\sharpE{H}$ satisfy the following contamination principle.
\[
\neg\omega=\omega\hbox{ and }a\circ\omega=\omega\hbox{ for all }a\in H\cup\{\omega\},
\]
where $\circ$ is any binary operator in $\mathcal{L}$.

We now consider the class of $\mathcal{L}$-matrices $\sharpE{\mathcal{H}}=\{\langle\sharpE{H},\{1,\omega\}\rangle\mid\,\langle\mathbf{H},\{1\}\rangle\in\mathcal{H}\}$ and semantically define IPWK as the logic $\left\langle\Fmbf,\paramodels\right\rangle$, where $\paramodels\,\subseteq\pow(\Fm)\times\Fm$ is defined as follows. For any $\Sigma\cup\{\alpha\}\subseteq\Fm$, $\Sigma\paramodels\alpha$ iff for every $\langle\sharpE{H},\{1,\omega\}\rangle\in\sharpE{\mathcal{H}}$ and every valuation $\val:\Fmbf\to\sharpE{H}$,
\[
\val[\Sigma]\subseteq\{1,\omega\}\hbox{ implies } \val(\alpha)\in\{1,\omega\}.
\]

The next theorems shows that IPWK is actually the left variable inclusion companion of IPC, that is, $\mathrm{IPWK}=\mathrm{IPC}^l$.

\begin{thm}\label{thm:pil-il_sem_conn}
  For all $\Sigma\cup\{\alpha\}\subseteq\Fm$, $\Sigma\paramodels\alpha$ if and only if there is a $\Delta\subseteq\Sigma$ such that $\var(\Delta)\subseteq\var(\alpha)$
  and $\Delta\inyields\alpha$. Moreover, since IPC is finitary, a finite such $\Delta\subseteq\Sigma$ can be found.
\end{thm}

\begin{proof}
  Suppose that $\Sigma\paramodels\alpha$.  Let $\Delta=
  \{\varphi\in\Sigma\mid\var(\varphi)\subseteq\var(\alpha)\}$. We will show that $\Delta\inmodels\alpha$. 
  
  Let $\mathbf{H}$ be a Heyting algebra and $v:\Fmbf\to\mathbf{H}$ be a valuation such that $v[\Delta]\subseteq\{1\}$. We note that $v[\Delta]=\emptyset$ iff $\Delta=\emptyset$, thus $\Delta\neq\emptyset$ implies $v[\Delta]=\{1\}$.
  
  Let $\sharpE{H}$ be the extended Heyting algebra corresponding to $\mathbf{H}$. We now construct the valuation $\val:\Fmbf\to\sharpE{H}$ by mapping variables as follows.
  \[
    \val(p)=\left\{\begin{array}{ll}
    v(p)&\hbox{if }p\in\var(\alpha)\\
    \omega&\hbox{if }p\notin\var(\alpha)
    \end{array}\right.
  \]
  If $\Sigma=\emptyset$, then $\paramodels\alpha$, which implies that $\val(\alpha)\in\{1,\omega\}$. Now clearly, $\val(\alpha)=v(\alpha)\in\mathbf{H}$. Hence $v(\alpha)=1$. Thus in this case, we have 
  \[
  \Delta=\emptyset\inmodels\alpha,\hbox{ that is, }\inmodels\alpha.
  \]
  Now suppose $\Sigma\neq\emptyset$ and $\varphi\in\Sigma$. Then we have the following cases.

  \textbf{Case 1}: $\var(\varphi)\subseteq\var(\alpha)$.

  In this case, $\varphi\in\Delta$ and $\val(\varphi)=v(\varphi)\in v[\Delta]$. So since $\Delta\neq\emptyset$ and hence $v[\Delta]=\{1\}$, we have $\val(\varphi)=1$.

  \textbf{Case 2}: $\var(\varphi)\cap\left(\var(\Sigma)\setminus\var(\alpha)\right)
  \neq\emptyset$.

  In this case, $\val(\varphi)=\omega$. 
  
  Thus we conclude that, for any $\varphi\in\Sigma$, $\val(\varphi)\in\{1,\omega\}$, that is, $\val[\Sigma]\subseteq\{1,\omega\}$. 
  
  Then since $\Sigma\paramodels\alpha$, $\val(\alpha)
  \in\{1,\omega\}$. Now $\val(\alpha)=v(\alpha)\in\mathbf{H}$. 
  
  Hence $v(\alpha)=1$. Since $\mathbf{H}$ was an arbitrarily chosen Heyting algebra and $v:\Fmbf\to\mathbf{H}$ was an arbitrarily chosen valuation with $v[\Delta]\subseteq\{1\}$, this proves that $\Delta\inmodels\alpha$. Then by the completeness of IPC with respect to valuations in Heyting algebras, $\Delta\inyields\alpha$.

  Conversely, suppose that there is some $\Delta\subseteq\Sigma$, with 
  $\var(\Delta)\subseteq\var(\alpha)$, such that $\Delta\inyields\alpha$. Thus $\Delta\inmodels\alpha$, by the soundness of IPC with respect to valuations in Heyting algebras. We need 
  to show that $\Sigma\paramodels\alpha$.

  Let $\sharpE{H}=\mathbf{H}\oplus\mathbf{1}$ be an extended Heyting algebra, and $\val:\Fmbf\to\sharpE{H}$ be a valuation such that $\val[\Sigma]\subseteq\{1,\omega\}$. 
  Following are the possible cases.

  \textbf{Case 1}: $\val(p)=\omega$ for some $p\in\var(\alpha)$. 

  Then $\val(\alpha)=\omega\in\{1,\omega\}$.

  \textbf{Case 2}: $\val(p)\neq\omega$ for all $p\in\var(\alpha)$.

  Let $a_0\in H$, the universe of $\mathbf{H}$. Then we construct a valuation 
  $v:\Fmbf\to\mathbf{H}$ by mapping the variables as follows.
  \[
    v(p)=\left\{\begin{array}{ll}
    \val(p)&\hbox{if }p\in\var(\alpha)\\
    a_0&\hbox{otherwise}
    \end{array}\right.
  \]
  Then since $\var(\Delta)\subseteq\var(\alpha)$, $v[\Delta]=\val[\Delta]
  \subseteq\val[\Sigma]\subseteq\{1,\omega\}$. 
  
  So $v[\Delta]\subseteq\{1,\omega\}\cap\mathbf{H}=\{1\}$. Since $\Delta\inmodels\alpha$, this implies that $v(\alpha)=1$. Now clearly, $\val(\alpha)=v(\alpha)$. Hence $\val(\alpha)=1\in\{1,\omega\}$. 
  
  This 
  proves that $\Sigma\paramodels\alpha$.
\end{proof}

\begin{rem}\label{rem:pil-il_taut}
Since IPWK is the left variable inclusion companion of IPC, we have $\paramodels\alpha$ iff $\inyields\alpha$, by Theorem \ref{lem:S-thm_iff_Sl-thm}.
\end{rem}

\subsubsection{An axiomatization of IPWK}
Let $\Fmbf$ be the formula algebra over a countable set of propositional variables $V$ of the same type, $\mathcal{L}$, as the Heyting algebras. We now introduce a Hilbert-style logic HIPWK.

\begin{dfn}
  HIPWK is the logic $\left\langle\Fmbf,\parayields\right\rangle$, where $\parayields$ is the substitution-invariant syntactic consequence relation of the deductive system with the following axioms and inference rule.
  \begin{enumerate}[label=A\arabic*.]
  \item $\alpha\limp(\beta\limp\alpha)$;
  \item $(\alpha\limp(\beta\limp\gamma))\limp((\alpha\limp\beta)\limp(\alpha\limp\gamma))$;
  \item $\alpha\limp(\beta\limp(\alpha\land\beta))$;
  \item $\alpha\land\beta\limp\alpha$;
  \item $\alpha\land\beta\limp\beta$;
  \item $\alpha\limp\alpha\lor\beta$;
  \item $\beta\limp\alpha\lor\beta$;
  \item $(\alpha\limp\gamma)\limp((\beta\limp\gamma)\limp(\alpha\lor\beta\limp\gamma))$;
  \item $(\alpha\limp\beta)\limp((\alpha\limp\lnot\beta)\limp\lnot\alpha)$;
  \item $0\limp\alpha$;
  \end{enumerate}
  
  \[
    \begin{array}{c}
    \alpha,\quad\alpha\limp\beta\\
    \hline
    \beta
    \end{array},
    \quad\hbox{provided }\var(\alpha)\subseteq\var(\beta)\quad[\hbox{Restricted Modus Ponens (RMP)}].
  \]
\end{dfn}  

\begin{rem}
The axioms (A1)--(A10) along with unrestricted modus ponens (MP) constitute a Hilbert-style presentation for IPC. Thus HIPWK is the restricted rules companion of IPC. Hence the following theorem is obtained as a corollary to Theorem \ref{thm:S-thm_iff_Sr-thm}.
\end{rem}

\begin{thm}\label{thm:pil_thm_iff_il_thm}
  For any $\varphi\in\Fm$, $\parayields\varphi$ if and only if $\inyields\varphi$.
\end{thm}

It is now shown below that HIPWK is complete relative to the following class of matrices. 
\[
\sharpE{\mathcal{H}}=\{\langle\sharpE{H},\{1,\omega\}\rangle\mid\,\langle\mathbf{H},\{1\}\rangle\in\mathcal{H}\}.
\]
This would then imply that IPWK = HIPWK, and that they are the semantic and syntactic presentations of the same logic.

\begin{thm}[Soundness]\label{thm:sound}
  For all $\Sigma\cup\{\alpha\}\subseteq\Fm$, if $\Sigma\parayields\alpha$, then $\Sigma\paramodels\alpha$.
\end{thm}

\begin{proof}
  We recall that HIPWK is the restricted rules companion, and IPWK is the left variable inclusion companion of IPC. Then this follows from Theorem \ref{thm:Sr_sub_Sl}.
\end{proof}

\begin{thm}[Completeness]\label{thm:complete}
  For all $\Sigma\cup\{\alpha\}\subseteq\Fm$, if $\Sigma\paramodels\alpha$, then $\Sigma\parayields\alpha$.
\end{thm}

\begin{proof}
  Suppose $\Sigma\paramodels\alpha$.
  
  Then by Theorem \ref{thm:pil-il_sem_conn}, there exists a finite $\Delta\subseteq\Sigma$, with $\var(\Delta)\subseteq\var(\alpha)$, such that $\Delta\inyields\alpha$.
  
  If $\Delta=\emptyset$, then $\inyields\alpha$. So by Theorem \ref{thm:pil_thm_iff_il_thm}, $\parayields\alpha$, which implies that $\Sigma\parayields\alpha$. Now suppose  $\emptyset\neq\Delta=\{\varphi_1,\ldots,\varphi_n\}$. Then by applying the Deduction theorem in IPC $n$ times, we have 
  \[
  \inyields(\varphi_1\limp(\ldots\limp(\varphi_n\limp\alpha)\ldots)).
  \]
  Therefore, by Theorem \ref{thm:pil_thm_iff_il_thm}, we have 
  \[
  \parayields(\varphi_1\limp(\ldots\limp(\varphi_n\limp\alpha)\ldots)).
  \]
  Now since $\var(\Delta)\subseteq\var(\alpha)$, by applying RMP $n$ times, we have $\Delta\parayields\alpha$. Finally, since $\Delta\subseteq\Sigma$, we have $\Sigma\parayields\alpha$.
\end{proof}

The above soundness and completeness theorems show that IPWK = HIPWK. Hence we can conclude that $\mathrm{IPC}^l=\mathrm{IPWK}=$HIPWK$=\mathrm{IPC}^{re}$.

\begin{rem}\label{rem:paraconsistency_IPWK}
IPWK is, at least, weakly paraconsistent since ECQ fails in IPWK as it is a left variable inclusion companion logic, by Theorem \ref{thm:Sl-paraconsistent}. Thus IPWK is, at least, weakly paraconsistent.

However, since LNC holds in IPC, by Theorem \ref{thm:pil_thm_iff_il_thm}, LNC holds in IPWK. Thus IPWK is not strongly paraconsistent.

It is interesting to note that $0\limp\alpha$ is an axiom of IPWK. This means that the converse of the Deduction theorem does not hold.
\end{rem}

However, a restricted converse of the Deduction theorem holds in HIPWK as indicated in part (ii) of the following theorem. 

\begin{thm}\label{thm:partialDT}
  For any $\Sigma\cup\{\alpha,\beta\}\subseteq\Fm$, we have the following.
  \begin{enumerate}[label=(\roman*)]
  \item If $\Sigma\cup\{\alpha\}\parayields\beta$ then $\Sigma\parayields\alpha\limp\beta$. (Deduction theorem)
  \item If $\Sigma\parayields\alpha\limp\beta$ and $\var(\alpha)\subseteq\var(\beta)$, then $\Sigma\cup\{\alpha\}\parayields\beta$. (Restricted converse of the Deduction theorem)
  \end{enumerate}
\end{thm}
  
\begin{proof}
  We note that the Deduction theorem and its converse hold in IPC. Since HIPWK = IPWK is the left variable inclusion companion of IPC, this follows from Theorems \ref{thm:Ded_Sl} and \ref{thm:Det_Sl}.
\end{proof}

\begin{rem}
Another interesting consequence of the restriction on the converse of the Deduction theorem is as follows. For any $\alpha,\beta\in\Fm$, $\alpha\parayields\beta$ implies that $\parayields\alpha\limp\beta$ but the converse does not hold, in general, thus breaking the classically perceived equivalence between entailment and implication.

Similar remarks can be made regarding PWK, some of which are discussed in \cite{CiuniCarrara2016,Bonzio2017}, and any logic of left variable inclusion where the Deduction theorem and its converse hold.
\end{rem}

\subsection{Paraconsistent pre-rough logic (PPRL)}\label{subsec:pprl}

\begin{dfn}\label{dfn:qba}
  Suppose $\mathcal{L}=\{\land,\lor,\neg,0,1\}$ is a logical language, where $\land,\lor,\neg,0,1$ are operators with arities $2,2,1,0,0$ respectively. Then an algebra $\mathbf{Q}$ of type $\mathcal{L}$, with universe $Q$, is called a \emph{quasi-Boolean algebra} if the following conditions are satisfied.
  \begin{enumerate}[label=(\roman*)]
      \item $\mathbf{Q}$ as an algebra of type $\mathcal{L}\setminus\{\neg\}$ is a bounded distributive lattice.
      \item $\neg\neg a=a$ for all $a\in Q$.
      \item $\neg(a\lor b)=\neg a\land\neg b$ for all $a,b\in Q$.
  \end{enumerate}
\end{dfn} 

\begin{rem}
The difference between a Boolean algebra and a quasi-Boolean algebra is that in the latter, it is not necessarily the case that 
${a\lor\neg a=1}$ or equivalently, ${a\land\neg a=0}$ for all $a$ in its universe.
\end{rem}

\begin{rem}
Quasi-Boolean algebras were named such and investigated by Bia\l{}ynicki-Birula and Rasiowa in \cite{Bialynicki-BirulaRasiowa1957}. Essentially identical structures, with no least element, were also studied by Moisil in \cite{Moisil1935} under the name \emph{de Morgan lattices}, and by Kalman in \cite{Kalman1958} under the name \emph{distributive i-lattices}. The above definition of a quasi-Boolean algebra can be found in \cite{Rasiowa1974}. For more on quasi-Boolean algebras and comparisons between quasi-Boolean algebras, De Morgan lattices, and distributive i-lattices, see \cite{Rasiowa1974, Gastaminza1968, Dunn1982}.
\end{rem}

\begin{dfn}\label{dfn:pra}
Suppose $\mathcal{L}=\{\land,\lor,\limp,\neg,I,C,0,1\}$ is a logical language, where the arities of $\land,\lor,\limp,\neg,I,C,0,1$ are $2,2,2,1,1,1,0,0$ respectively. Then an algebra $\mathbf{R}$ of type $\mathcal{L}$, with universe $R$, is called a \emph{pre-rough algebra} if the following conditions hold.
\begin{enumerate}[label=(\roman*)]
    \item $\mathbf{R}$ as an algebra of type $\mathcal{L}\setminus\{\limp,I,C\}$ is a quasi-Boolean algebra.
    \item $I1=1$.
    \item $I(a\land b)=Ia\land Ib$ for all $a,b\in R$.
    \item $\neg Ia\lor Ia=1$ for all $a\in R$.
    \item $Ia\limp a=1$ for all $a\in R$.
    \item $Ca=\neg I\neg a$ for all $a\in R$.
    \item $a\limp b=(\neg Ia\lor Ib)\land(\neg Ca\lor Cb)$ for all $a,b\in R$.
    \item $Ca\limp Cb=1$ and $Ia\limp Ib=1$ imply $a\limp b=1$ for all $a,b\in R$.
\end{enumerate}
\end{dfn}

\begin{rem}
A pre-rough algebra was first defined in \cite{BanerjeeChakraborty1996}. This was streamlined later in \cite{ChakrabortySahaSen2014}. Both these definitions used an order relation as a primitive. The above definition of a pre-rough algebra, although equivalent, is slightly different from both the original and the streamlined definitions.
\end{rem}

Let $\mathcal{L}$ be the language described in Definition \ref{dfn:pra}, $V$ a countable set of propositional variables, and $\Fmbf$ denote the formula algebra over $V$ of type $\mathcal{L}$. 

Thus pre-rough algebras are algebras of type $\mathcal{L}$. Pre-rough logic (PRL) can then be described semantically using the class of $\mathcal{L}$-matrices,
\[
\mathcal{R}=\{\langle\mathbf{R},\{1\}\rangle\mid\,\mathbf{R}\hbox{ is a pre-rough algebra}\},
\]
as the logic $\left\langle\Fmbf,\prmodels\right\rangle$, where $\prmodels\,\subseteq\mathcal{P}(\Fm)\times\Fm$ is defined as follows. For any $\Sigma\cup\{\alpha\}\subseteq\Fm$, 
\[
\Sigma\prmodels\alpha\hbox{ iff for all }\langle\mathbf{R},\{1\}\rangle\in\mathcal{R}\hbox{ and for all valuations }v:\Fmbf\to\mathbf{R},v[\Sigma]\subseteq\{1\}\hbox{ implies }v(\alpha)=1.
\]

Now, the trivial $\mathcal{L}$-matrix, $\langle\mathbf{1},\{\omega\}\rangle\in\mathcal{R}$. So by Theorem \ref{thm:Plonka_complete}, the left variable inclusion companion of PRL is complete with respect to $\Pl(\mathcal{R})$, the class of P\l{}onka sums of the directed systems of matrices in $\mathcal{R}$.

In particular, for each pre-rough algebra $\mathbf{R}$, the P\l{}onka sum of the matrices $\langle\mathbf{R},\{1\}\rangle$ and $\langle\mathbf{1},\{\omega\}\rangle$, $\langle\mathbf{R}\oplus\mathbf{1},\{1,\omega\}\rangle\in\Pl(\mathcal{R})$. For any pre-rough algebra $\mathbf{R}$, we will refer to the algebra $\mathbf{R}\oplus\mathbf{1}$ as the \emph{extended pre-rough algebra} corresponding to $\mathbf{R}$, and denote it by $\sharpE{R}$. We note that the operations in $\sharpE{R}$ satisfy the following contamination principle.
\[
\neg\omega=\omega,\,I\omega=\omega,\,C\omega=\omega,\hbox{ and } a\circ\omega=\omega\hbox{ for all }a\in\sharpE{R},
\]
where $\circ$ denotes any binary operator in $\mathcal{L}$.

We now consider the class of $\mathcal{L}$-matrices $\sharpE{\mathcal{R}}=\{\langle\sharpE{R},\{1,\omega\}\mid\,\langle\mathbf{R},\{1\}\rangle\in\mathcal{R}\}$ and semantically define PPRL as the logic $\langle\Fmbf,\pprmodels\rangle$, where $\pprmodels\,\subseteq\mathcal{P}(\Fm)\times\Fm$ is defined as follows. For any $\Sigma\cup\{\alpha\}\subseteq\Fm$, $\Sigma\pprmodels\alpha$ iff for every $\langle\sharpE{R},\{1,\omega\}\rangle\in\sharpE{\mathcal{R}}$ and every valuation $\val:\Fmbf\to\sharpE{R}$,
\[
\val[\Sigma]\subseteq\{1,\omega\}\hbox{ implies } \val(\alpha)\in\{1,\omega\}.
\] 

The next theorem shows that PPRL is actually the left variable inclusion companion of PRL, that is, $\mathrm{PPRL}=\mathrm{PRL}^l$.

\begin{thm}\label{thm:prl-pprl_sem_conn}
  For all $\Sigma\cup\{\alpha\}\subseteq\Fm$, $\Sigma\pprmodels\alpha$ if and only if there is a $\Delta\subseteq\Sigma$ such that $\var(\Delta)\subseteq\var(\alpha)$
  and $\Delta\prmodels\alpha$. Moreover, since PRL is finitary, a finite such $\Delta\subseteq\Sigma$ can be found.
\end{thm}

\begin{proof}
The proof of this theorem is essentially the same as the proof of Theorem \ref{thm:pil-il_sem_conn} with Heyting algebras, extended Heyting algebras, $\paramodels$, and $\inmodels$ replaced by pre-rough algebras, extended pre-rough algebras, $\pprmodels$, and $\prmodels$, respectively. The only difference in this case is that we do not move between the syntactic and semantic consequence relations as PRL is only defined semantically.
\end{proof}

\begin{rem}
Since PPRL is the left variable inclusion companion of PRL, we have $\pprmodels\alpha$ iff $\prmodels\alpha$, for any $\alpha\in\Fm$, by Lemma \ref{lem:S-thm_iff_Sl-thm}.
\end{rem}

\subsubsection{An axiomatization of PPRL}

Let $\Fmbf$ be the formula algebra over a countable set of propositional variables $V$ of the same type, $\mathcal{L}$, as the pre-rough algebras. We first describe a Hilbert-style logic HPRL.

\begin{dfn}\label{dfn:hprl}
HPRL is the logic $\langle\Fmbf,\pryields\rangle$, where $\pryields$ is the substitution-invariant syntactic consequence relation of the deductive system with the following axioms and inference rules.

\textbf{Axioms:}
\begin{enumerate}[label=A\arabic*.]
    \item $\alpha\limp\neg\neg\alpha$
    \item $\neg\neg\alpha\limp\alpha$
    \item $(\alpha\land\beta)\limp\beta$
    \item $(\alpha\land\beta)\limp(\beta\land\alpha)$
    \item $(\alpha\land(\beta\lor\gamma))\limp((\alpha\land\beta)\lor(\alpha\land\gamma))$
    \item $((\alpha\land\beta)\lor(\alpha\land\gamma))\limp(\alpha\land(\beta\lor\gamma))$
    \item $(\alpha\lor\beta)\limp\neg(\neg\alpha\land\neg\beta)$
    \item $\neg(\neg\alpha\land\neg\beta)\limp(\alpha\lor\beta)$
    \item $C\alpha\limp\neg I\neg\alpha$
    \item $\neg I\neg\alpha\limp C\alpha$
    \item $I\alpha\limp\alpha$
    \item $(I\alpha\land I\beta)\limp I(\alpha\land\beta)$
    \item $(\alpha\limp\beta)\limp((\neg I\alpha\lor I\beta)\land(\neg C\alpha\lor C\beta))$
    \item $((\neg I\alpha\lor I\beta)\land(\neg C\alpha\lor C\beta))\limp(\alpha\limp\beta)$
\end{enumerate}

\textbf{Rules of inference:}
\begin{enumerate}[label=R\arabic*.]
    \item $\begin{array}{c}
    \alpha,\quad\alpha\limp\beta\\
    \hline
    \beta
    \end{array}
    \quad$ [Modus ponens (MP)]
    \item $\begin{array}{c}
    \alpha\limp\beta,\quad\beta\limp\gamma\\
    \hline
    \alpha\limp\gamma
    \end{array}
    \quad$ [Hypothetical Syllogism (HS)]
    \item $\begin{array}{c}
    \alpha\\
    \hline
    \beta\limp\alpha
    \end{array}$
    
    \item $\begin{array}{c}
    \alpha\limp\beta\\
    \hline
    \neg\beta\limp\neg\alpha
    \end{array}$
    
    \item $\begin{array}{c}
    \alpha\limp\beta,\quad\alpha\limp\gamma\\
    \hline
    \alpha\limp(\beta\land\gamma)
    \end{array}$
    
    \item $\begin{array}{c}
    \alpha\limp\beta,\quad\beta\limp\alpha,\quad\gamma\limp\delta,\quad\delta\limp\gamma\\
    \hline
    (\alpha\limp\gamma)\limp(\beta\limp\delta)
    \end{array}$
    
    \item $\begin{array}{c}
    \alpha\limp\beta\\
    \hline
    I\alpha\limp I\beta
    \end{array}$
    
    \item $\begin{array}{c}
    \alpha\\
    \hline
    I\alpha
    \end{array}$
    
    \item $\begin{array}{c}
    I\alpha\limp I\beta,\quad C\alpha\limp C\beta\\
    \hline
    \alpha\limp\beta
    \end{array}$
\end{enumerate}
\end{dfn}

\begin{rem}
HPRL is an axiomatization of PRL, in the sense that HPRL is sound and weakly complete with respect to valuations in pre-rough algebras. That is, for any $\Sigma\cup\{\alpha\}\subseteq\Fm$,
\[
\Sigma\pryields\alpha\hbox{ implies }\Sigma\prmodels\alpha\quad\hbox{and}\quad\prmodels\alpha\hbox{ implies }\pryields\alpha.
\]
A Hilbert-style presentation for PRL was first given in \cite{BanerjeeChakraborty1996}. This was later streamlined in \cite{ChakrabortySahaSen2014}. The soundness and weak completeness of these Hilbert systems with respect to valuations in pre-rough algebras were also discussed in the above mentioned papers. Our set of axioms is, however, slightly modified from the ones in the aforementioned papers.
\end{rem}

We next introduce the restricted rules companion of HPRL. One can easily see from Definition \ref{dfn:hprl} that, if $\dfrac{\Gamma}{\alpha}$ is one of the rules (R3) -- (R9), then $\var(\Gamma)\subseteq\var(\alpha)$, i.e. there is no loss of variables in passing from the premises to the conclusions of these rules. Hence imposing variable inclusion restrictions on these rules do not result in anything different. Hence HPRL and $\HPPRL$ will differ only in the restrictions of MP and HS. The following is a Hilbert-style presentation of $\HPPRL$.

\begin{dfn}\label{dfn:hpprl}
$\HPPRL$ is the Hilbert-style logic $\langle\Fmbf,\ppryields\rangle$, where $\Fmbf$ is the same as in the definition of HPRL, and $\ppryields$ is the substitution-invariant syntactic consequence relation of the deductive system with the same set of axioms as HPRL and the following inference rules.

\begin{enumerate}[label=R\arabic*$^\prime$.]
    \item $\begin{array}{c}
    \alpha,\quad\alpha\limp\beta\\
    \hline
    \beta
    \end{array},
    \quad\hbox{provided }\var(\alpha)\subseteq\var(\beta)\quad$ [Restricted MP (RMP)]
    
    \item $\begin{array}{c}
    \alpha\limp\beta,\quad\beta\limp\gamma\\
    \hline
    \alpha\limp\gamma
    \end{array},
    \quad\hbox{provided }\var(\beta)\subseteq\var(\alpha)\cup\var(\gamma)\quad$ [Restricted HS (RHS)]
    \end{enumerate}
    (R3) -- (R9) as in Definition \ref{dfn:hprl}.
\end{dfn}

\begin{thm}\label{thm:pprl_thm_iff_prl_thm}
  For any $\varphi\in\Fm$, $\ppryields\varphi$ if and only if $\pryields\varphi$.
\end{thm}

\begin{proof}
  This follows from Theorem \ref{thm:S-thm_iff_Sr-thm}.
\end{proof}

\begin{rem}\label{rem:HPRL_SrXSl}
It may be noted here that $\HPPRL$ is not the left variable inclusion companion of HPRL. This can be shown via the following derivation in HPRL.

Suppose $p,q$ are distinct variables. Then
\[
\begin{array}{lcll}
     p\land q&\pryields&1.\;p\land q&\hbox{(Hypothesis)}\\
     &&2.\;(p\land q)\limp q&\hbox{(Axiom (A3))}\\
     &&3.\;q&\hbox{(MP on (1) and (2))}\\
     &&4.\;p\limp q&\hbox{(Rule (R3) on (3))}
\end{array}
\]
Thus $p\land q\pryields p\limp q$. Since $\var(p\land q)=\var(p\limp q)$, $p\land q\pryields^l p\limp q$, where $\pryields^l$ is the consequence relation in the left variable inclusion companion of HPRL. However, the above derivation cannot be carried out in $\HPPRL$ because the application of MP in Step 3 cannot be replaced by an application of RMP. 

More generally, any derivation of $p\limp q$ from $p\land q$ that involves splitting $p\land q$ to either $p$ or $q$ cannot be done in $\HPPRL$. This is because $p\land q\not\pryields^l p,q$ and hence $p\land q\not\ppryields p,q$ as $\ppryields\,\subseteq\,\pryields^l$ by Theorem \ref{thm:Sr_sub_Sl}. Thus a derivation of $p\limp q$ from $p\land q$ in $\HPPRL$ is possible only if this can be achieved in HPRL without splitting the conjunction. Now, scanning through all combinations of axioms of rules of HPRL that can lead to the derivation of an implication from a conjunction (this is possible since there are finitely many such), we find that such a derivation is not possible. In each case, the problem of arriving at an implication, such as $p\limp q$, from $p\land q$ gets reduced to a problem of deriving a similar or more complex implication at an earlier step in the derivation. Thus we can conclude that $p\land q\not\ppryields p\limp q$.

Hence by Theorem \ref{thm:deduction_thm-Sr=Sl}, we can conclude that the Deduction theorem does not hold in HPRL.

The above example serves to illustrate again that the restricted rules companion is not always the same as the left variable inclusion companion of a logic.
\end{rem}

\begin{thm}[Soundness]\label{thm:sound_pprl}
  For all $\Sigma\cup\{\alpha\}\subseteq\Fm$, if $\Sigma\ppryields\alpha$, then $\Sigma\pprmodels\alpha$.
\end{thm}

\begin{proof}
Suppose $\Sigma\ppryields\alpha$.

If $\Sigma=\emptyset$, then $\ppryields\alpha$. Now since $\HPPRL$ is the restricted rules companion of HPRL, by Theorem \ref{thm:S-thm_iff_Sr-thm}, we have $\pryields\alpha$. Then by the soundness of HPRL relative to PRL, $\prmodels\alpha$. Finally, using the fact that PPRL is the left variable inclusion companion of PRL, by Lemma \ref{lem:S-thm_iff_Sl-thm}, it can be concluded that $\pprmodels\alpha$. 

Thus in particular, if $\alpha$ is an axiom of $\HPPRL$, then $\pprmodels\alpha$.

Now, let $\Sigma\neq\emptyset$.

First suppose that $\dfrac{\Sigma}{\alpha}$ is an instance of one of the rules of inference of $\HPPRL$. Then $\dfrac{\Sigma}{\alpha}$ is also an instance of a rule in HPRL. So, by soundness of HPRL relative to PRL, we have $\Sigma\prmodels\alpha$. Now, since $\dfrac{\Sigma}{\alpha}$ is an instance of a rule of inference of $\HPPRL$, $\var(\Sigma)\subseteq\var(\alpha)$. Thus $\Sigma\pprmodels\alpha$ as PPRL is the left variable inclusion companion of PRL.

Finally, we use induction to prove the general case as follows.

Let $D=\langle\varphi_1,\ldots,\varphi_n(=\alpha)\rangle$ be a derivation of $\alpha$ from $\Sigma$ in $\HPPRL$.
  
Suppose $\langle\sharpE{R},\{1,\omega\}\rangle\in\sharpE{\mathcal{R}}$, and $\val:\Fmbf\to\sharpE{R}$ be a valuation such that $\val[\Sigma]\subseteq\{1,\omega\}$. Since $\Sigma\neq\emptyset$, $\val[\Sigma]\neq\emptyset$.
  
Now, each $\varphi_i, 1\le i\le n$, in $D$ is an instance of an axiom of $\HPPRL$, or is an element of $\Sigma$, or is obtained by applying one of the rules of inference stated in Definition \ref{dfn:hpprl} on a subset of $\{\varphi_1,\ldots,\varphi_{i-1}\}$. We will show that $\val(\varphi_i)\in\{1,\omega\}$ for each $1\le i\le n$, by induction on $i$.

\textbf{Base case}: $i=1$

Then $\varphi_i=\varphi_1$ is either an axiom of $\HPPRL$ or a member of $\Sigma$. If $\varphi_1$ is an axiom, then as previously shown, $\pprmodels\varphi$, that is, $\val(\varphi_i)\in\{1,\omega\}$. On the other hand, if $\varphi_1\in\Sigma$, then $\val(\varphi_1)\in\{1,\omega\}$ by our assumption.

\textbf{Induction hypothesis}: Suppose $\emptyset\neq\val\left[\{\varphi_1,\ldots,\varphi_{i-1}\}\right]\subseteq\{1,\omega\}$ for some $1<i\le n$.

\textbf{Induction step}: We need to show that $\val(\varphi_i)\in\{1,\omega\}$. If $\varphi_i$ is an axiom or a member of $\Sigma$, then by the same reasoning as in the base case, we have $\val(\varphi_i)\in\{1,\omega\}$.
So suppose $\varphi_i$ is obtained from a subset of $\{\varphi_1,\ldots,\varphi_{i-1}\}$ by an application of one of the rules of inference in Definition \ref{dfn:hpprl}. Then, by the induction hypothesis, $\val(\varphi_j)\in\{1,\omega\}$ for each $1\le j\le i-1$. So by our previous arguments, $\val(\varphi_i)\in\{1,\omega\}$. 
  
Thus $\val(\varphi_i)\in\{1,\omega\}$ for each $1\le i\le n$. Hence $\val(\alpha)=\val(\varphi_n)\in\{1,\omega\}$. This implies that $\Sigma\pprmodels\alpha$, since $\langle\sharpE{R},\{1,\omega\}\rangle\in\sharpE{\mathcal{R}}$ and $\val:\Fmbf\to\sharpE{R}$ were arbitrary.
\end{proof}

\begin{thm}[Weak Completeness]\label{thm:thm-complete_pprl}
  For any $\varphi\in\Fm$, if $\pprmodels\varphi$ then $\ppryields\varphi$.
\end{thm}

\begin{proof}
    This follows from the facts that PPRL is the left variable inclusion companion of PRL, Lemma \ref{lem:S-thm_iff_Sl-thm}, the weak completeness of PRL with respect to valuations in pre-rough algebras \cite{BanerjeeChakraborty1996,ChakrabortySahaSen2014}, and Theorem \ref{thm:pprl_thm_iff_prl_thm}.
\end{proof}

\begin{rem}
Summarizing the above results, we can say that the soundness and weak completeness of the restricted rules companion of HPRL relative to the left variable inclusion companion of PRL follow from the soundness and weak completeness of HPRL relative to PRL. However, the restricted rules companion of HPRL does not coincide with the left variable inclusion companion of HPRL.
\end{rem}

\begin{rem}\label{rem:paraconsistent_PPRL}
ECQ fails in PPRL since it is the left variable inclusion companion of PRL, by Theorem \ref{thm:Sl-paraconsistent}. ECQ fails in $\HPPRL$ as it is the restricted rules companion of HPRL, by Corollary \ref{cor:Sr-paraconsistent}.
\end{rem}

\begin{rem}
$\HPPRL$ and PPRL are, in fact, strongly paraconsistent.

In addition to the failure of ECQ, as shown above, the following arguments prove that LNC also fails in these two logics.

The algebra $\mathbf{R}$, of type $\mathcal{L}=\{\land,\lor,\limp,\neg,I,C,0,1\}$, with universe $R=\{0,a,1\}$, where the operations are defined according to the following tables, is mentioned in \cite{BanerjeeChakraborty1996} as the smallest non-trivial pre-rough algebra.
\[
\begin{array}{cccc}
     \begin{array}{c|c|c|c}
          &\neg&I&C\\
          \hline
          0&1&0&0\\
          a&a&a&a\\
          1&0&1&1
     \end{array}&\begin{array}{c|c|c|c}
          \land&0&a&1\\
          \hline
          0&0&0&0\\
          a&0&a&a\\
          1&0&a&1
     \end{array}&\begin{array}{c|c|c|c}
          \lor&0&a&1\\
          \hline
          0&0&a&1\\
          a&a&a&1\\
          1&1&1&1
     \end{array}&\begin{array}{c|c|c|c}
          \limp&0&a&1\\
          \hline
          0&1&1&1\\
          a&a&a&1\\
          1&0&a&1
     \end{array}
\end{array}
\]
Clearly, in the above pre-rough algebra, $\neg(a\land\neg a)\neq1$. Thus by the soundness of HPRL with respect to valuations in pre-rough algebras, LNC fails in PRL. Then the claim follows by Remark \ref{rem:Sl&Sr_paracons}.
\end{rem}

\section{``Other'' paraconsistent logics}\label{sec:paracons_other}

In Section \ref{sec:lvari}, it was proved that the left variable inclusion companion of any logic is paraconsistent, at least in the weak sense (Theorem \ref{thm:Sl-paraconsistent}). As a result of this, the restricted rules companion of any logic with a Hilbert-style presentation is also, at least, weakly paraconsistent (Corollary \ref{cor:Sr-paraconsistent}). We now argue in this section, using two examples, that not all paraconsistent logics can be characterized as the left variable inclusion companion or the restricted rules companion of a logic. The following two logics, $\rml$ and $\lps$, are both weakly paraconsistent and have similar looking semantics, but while the Deduction theorem holds in $\lps$, it does not in $\rml$.

\subsection{\texorpdfstring{$\rml$}{RM3}}
Our first example is the logic $\rml$ (relevance-mingle logic). This is one of the relevance logics proposed in \cite{AndersonBelnap1975} that has been studied as an important 3-valued paraconsistent logic in \cite{Avron1991,ArieliAvron2015}.  The logical language used for $\rml$ is $\mathcal{L}=\{\land,\lor,\limp,\neg\}$. Then with the usual assumption of a countable set of variables and the construction of a formula algebra of type $\mathcal{L}$, the logic $\rml$ can be described syntactically as the Hilbert-style logic $\langle\Fmbf,\vdash_{\rml}\rangle$ with the following axioms and rules.

\textbf{Axioms:}
\begin{enumerate}[label=A\arabic*.]
\item $\alpha\limp\alpha$
\item $(\alpha\limp\beta)\limp((\beta\limp\gamma)\limp(\alpha\limp\gamma))$
\item $\alpha\limp((\alpha\limp\beta)\limp\beta)$
\item $(\alpha\limp(\alpha\limp\beta))\limp(\alpha\limp\beta)$
\item $(\alpha\land\beta)\limp\alpha$
\item $(\alpha\land\beta)\limp\beta$
\item $((\alpha\limp\beta)\land(\alpha\limp\gamma))\limp(\alpha\limp(\beta\land\gamma))$
\item $\alpha\limp(\alpha\lor\beta)$
\item $\beta\limp(\alpha\lor\beta)$
\item $((\alpha\limp\gamma)\land(\beta\limp\gamma))\limp((\alpha\lor\beta)\limp\gamma)$
\item $(\alpha\land(\beta\lor\gamma))\limp((\alpha\land\beta)\lor(\alpha\land\gamma))$
\item $\neg\neg\alpha\limp\alpha$
\item $(\alpha\limp\neg\beta)\limp(\beta\limp\neg\alpha)$
\item $\alpha\limp(\alpha\limp\alpha)$
\item $\alpha\lor(\alpha\limp\beta)$
\end{enumerate}

\textbf{Rules of inference:}
\begin{enumerate}[label=R\arabic*.]
    \item $\begin{array}{c}
    \alpha,\quad\beta\\
    \hline
    \alpha\land\beta
    \end{array}
    \quad$ [$\land I$]
    \item $\begin{array}{c}
    \alpha,\quad\alpha\limp\beta\\
    \hline
    \beta
    \end{array}
    \quad$ [MP]
\end{enumerate}

$\rml$ is sound and complete with respect to the matrix $\langle M_3,\{1,1/2\}\rangle$, where $M_3$ is the algebra whose universe is $\{1,1/2,0\}$ and whose operations are given by the following tables  \cite{AndersonBelnap1975}.

\[
\begin{array}{llll}
     \begin{array}{|c|ccc|}
     \hline
          \land&1&1/2&0\\
          \hline
          1&1&1/2&0\\
          1/2&1/2&1/2&0\\
          0&0&0&0\\
          \hline
     \end{array}&
     \begin{array}{|c|ccc|}
     \hline
          \lor&1&1/2&0\\
          \hline
          1&1&1&1\\
          1/2&1&1/2&1/2\\
          0&1&1/2&0\\
          \hline
     \end{array}&
     \begin{array}{|c|ccc|}
     \hline
          \limp&1&1/2&0\\
          \hline
          1&1&0&0\\
          1/2&1&1/2&0\\
          0&1&1&1\\
          \hline
     \end{array}&
     \begin{array}{|c|c|}
     \hline
          \neg&\\
          \hline
          1&0\\
          1/2&1/2\\
          0&1\\
          \hline
     \end{array}
\end{array}
\]

\begin{rem}\label{rem:RM3_not-Sr_not-Sl}
It is clear from the above Hilbert-style presentation that $\rml$ is not the restricted rules companion of any logic since unrestricted modus ponens is a rule of inference in this logic.

$\rml$ cannot be the left variable inclusion companion of any logic either. To see this, suppose the contrary, that is, $\rml$ is the left variable inclusion companion of some logic $\langle\Fmbf,\vdash\rangle$. Suppose $p,q$ are distinct variables. Clearly, $\not\vdash_{\rml} q$. Now, by MP, we have $\{p,p\limp q\}\vdash_{\rml}q$. Then since $\rml$ is the left variable inclusion companion of $\langle\Fmbf,\vdash\rangle$, there must exist a $\Gamma\subseteq\{p,p\limp q\}$ such that $\var(\Gamma)\subseteq\{q\}$ and $\Gamma\vdash q$. However, since $p\neq q$, $\Gamma$ can only be the empty set. Thus $\vdash q$, and hence by Lemma \ref{lem:S-thm_iff_Sl-thm}, $\vdash_{\rml}q$, which is contrary to our assumption.

It is, however, well known that $\rml$ is paraconsistent. To see this, suppose $p,q$ are distinct variables and let $v:\Fmbf\to M_3$ be a valuation such that $v(p)=1/2$ and $v(q)=0$. Then $v(p)=v(\neg p)=1/2$ but $v(q)\notin\{1,1/2\}$. Hence by soundness, $\{p,\neg p\}\not\vdash_{\rml}q$.
\end{rem}

\begin{rem}\label{rem:non-Sr_non-Sl_paracons}
The arguments in the above remark can be generalized to any paraconsistent logic $\langle\Fmbf,\vdash\rangle$ (such that there are at least two distinct variables, and there exists a variable $q$ with $\not\vdash q$) that has the unrestricted modus ponens as a rule. Such a logic cannot be the restricted rules companion or the left variable inclusion companion of any logic.
\end{rem}

\begin{rem}\label{rem:RML_SrXSl}
The left variable inclusion and the restricted rules companions of $\rml$ do not coincide. This can be seen from the following example.

Suppose $p,q$ are distinct variables. Then
\[
\begin{array}{lcll}
     p\land q&\vdash_{\rml}&1.\;p\land q&\hbox{(Hypothesis)}\\
     &&2.\;(p\land q)\limp p&\hbox{(Axiom (A5))}\\
     &&3.\;p&\hbox{(MP on (1) and (2))}\\
     &&4.\;p\limp (p\lor q)&\hbox{(Axiom (A8))}\\
     &&5.\;p\lor q&\hbox{(MP in (3) and (4)}
\end{array}
\]
Thus $p\land q\vdash_{\rml} p\lor q$. Since $\var(p\land q)=\var(p\lor q)$, $p\land q\vdash_{\rml}^l p\lor q$, where $\vdash_{\rml}^l$ is the consequence relation in the left variable inclusion companion of $\rml$. However, the above derivation cannot be carried out in the restricted rules companion of $\rml$ because the application of MP in Step 3 cannot be replaced by an application of the restricted version of it. In fact, by reasoning similar to that in Remark \ref{rem:HPRL_SrXSl}, we can conclude that $p\lor q$ cannot be derived from $p\land q$ in the restricted rules companion of $\rml$.

Hence by Theorem \ref{thm:deduction_thm-Sr=Sl}, the Deduction theorem does not hold in $\rml$. This can also be shown via the following example.

Suppose $p,q$ are distinct variables, as before. Then using the tables for $M_3$ above, we see that for any valuation $v:\Fmbf\to M_3$, $v(p\lor\neg p),v(q\lor\neg q)\in\{1,1/2\}$. Hence by the completeness of $\rml$ with respect to the matrix $\langle M_3,\{1,1/2\}\rangle$, $p\lor\neg p\vdash_{\rml}q\lor\neg q$. Now, let $v^\prime:\Fmbf\to M_3$ be a valuation such that $v(p)=1$ and $v(q)=1/2$. Then using the tables for $M_3$ above, we see that $v^\prime(p\lor \neg p)=1$ and $v^\prime(q\lor\neg q)=1/2$, and hence $v^\prime((p\lor\neg p)\limp(q\lor\neg q))=0$. Thus by the soundness of $\rml$ with respect to the matrix $\langle M_3,\{1,1/2\}\rangle$, $\not\vdash_{\rml}(p\lor\neg p)\limp(q\lor\neg q)$.

The above example serves to illustrate again that the restricted rules companion is not always the same as the left variable inclusion companion of a logic.
\end{rem}

\subsection{\texorpdfstring{$\lps$}{LPS3}}

Our second example is the logic $\lps$ introduced in \cite{TarafderChakraborty2015}.  The logical language used for $\lps$ is the same as in $\rml$. With the usual assumption of a countable set of variables and the construction of a formula algebra of type $\mathcal{L}$, the logic $\lps$ can be described syntactically as the Hilbert-style logic $\langle\Fmbf,\vdash_{\lps}\rangle$ with the following axioms and rules \cite{TarafderChakraborty2015}.

\textbf{Axioms:}
\begin{enumerate}[label=A\arabic*.]
\item $\alpha\limp(\beta\limp\alpha)$
\item $(\alpha\limp(\beta\limp\gamma))\limp((\alpha\limp\beta)\limp(\alpha\limp\gamma))$
\item $(\alpha\land\beta)\limp\alpha$
\item $(\alpha\land\beta)\limp\beta$
\item $\alpha\limp(\alpha\lor\beta)$
\item $((\alpha\limp\gamma)\land(\beta\limp\gamma))\limp((\alpha\lor\beta)\limp\gamma)$
\item $((\alpha\limp\beta)\land(\alpha\limp\gamma))\limp(\alpha\limp(\beta\land\gamma))$
\item $(\alpha\limp\neg\neg\alpha)\land(\neg\neg\alpha\limp\alpha)$
\item $(\neg(\alpha\land\beta)\limp(\neg\alpha\lor\neg\beta))\land((\neg\alpha\lor\neg\beta)\limp\neg(\alpha\land\beta))$
\item $(\alpha\land\neg\alpha)\limp(\neg(\beta\limp\alpha)\limp\gamma)$
\item $(\alpha\limp\beta)\limp(\neg(\alpha\limp\gamma)\limp\beta)$
\item $(\neg\alpha\limp\beta)\limp(\neg(\gamma\limp\alpha)\limp\beta)$
\item $\bot\limp\alpha$
\item $(\alpha\land(\beta\limp\bot))\limp\neg(\alpha\limp\beta)$
\item $(\alpha\land(\neg\alpha\limp\bot))\lor(\alpha\land\neg\alpha)\lor(\neg\alpha\land(\alpha\limp\bot))$
\end{enumerate}

In the above axioms, $\bot$ is an abbreviation for $\neg(\varphi\limp\varphi)$, where $\varphi$ is any formula.

\textbf{Rules of inference:}
\begin{enumerate}[label=R\arabic*.]
    \item $\begin{array}{c}
    \alpha,\quad\beta\\
    \hline
    \alpha\land\beta
    \end{array}
    \quad$ [$\land I$]
    \item $\begin{array}{c}
    \alpha,\quad\alpha\limp\beta\\
    \hline
    \beta
    \end{array}
    \quad$ [MP]
\end{enumerate}

It has also been shown in \cite{TarafderChakraborty2015} that $\lps$ is sound and weakly complete with respect to the matrix $PS_3=\langle P,\{1,1/2\}\rangle$, where $P$ is the algebra whose universe is $\{1,1/2,0\}$ and whose operations are given by the following tables.
\[
\begin{array}{llll}
     \begin{array}{|c|ccc|}
     \hline
          \land&1&1/2&0\\
          \hline
          1&1&1/2&0\\
          1/2&1/2&1/2&0\\
          0&0&0&0\\
          \hline
     \end{array}&
     \begin{array}{|c|ccc|}
     \hline
          \lor&1&1/2&0\\
          \hline
          1&1&1&1\\
          1/2&1&1/2&1/2\\
          0&1&1/2&0\\
          \hline
     \end{array}&
     \begin{array}{|c|ccc|}
     \hline
          \limp&1&1/2&0\\
          \hline
          1&1&1&0\\
          1/2&1&1&0\\
          0&1&1&1\\
          \hline
     \end{array}&
     \begin{array}{|c|c|}
     \hline
          \neg&\\
          \hline
          1&0\\
          1/2&1/2\\
          0&1\\
          \hline
     \end{array}
\end{array}
\]

\begin{rem}
The semantics for the logics $\lps$ and $\rml$ differ only in the interpretation of the $\limp$ operator.
\end{rem}

\begin{rem}
$\lps$ cannot be the restricted rules companion or the left variable inclusion companion of any logic for the same reasons as in the case of $\rml$, that were mentioned in the Remarks \ref{rem:RM3_not-Sr_not-Sl} and \ref{rem:non-Sr_non-Sl_paracons}.

However, as noted in \cite{TarafderChakraborty2015}, $\lps$ is paraconsistent. To see this, suppose $p,q$ are distinct variables and $v:\Fmbf\to P$ is a valuation such that $v(p)=1/2$ and $v(q)=0$. Then clearly $v(p)=v(\neg p)=1/2$ but $v(q)\notin\{1,1/2\}$. Hence by Soundness, $\{p,\neg p\}\not\vdash_{\lps}q$, which implies that $\lps$ is paraconsistent.
\end{rem}

\begin{rem}
It has been proved in \cite{TarafderChakraborty2015} that the Deduction theorem holds in $\lps$. Hence by Theorem \ref{thm:deduction_thm-Sr=Sl}, the restricted rules companion of $\lps$ will coincide with its left variable inclusion companion.
\end{rem}

The above logics are just a couple among many paraconsistent logics that have unrestricted modus ponens and hence cannot be in the class of paraconsistent logics that are restricted rules companions or left variable inclusion companions of other logics. More examples can be found in \cite{ArieliAvron2015, DuttaChakraborty2015,TarafderChakraborty2015}. In addition to examples, \cite{DuttaChakraborty2015} also contains a recipe for creating new paraconsistent logics with consequence relations obeying certain conditions.

\section{Conclusions and future directions}

In this paper, we have proved the following new results.

\begin{itemize}
    \item The left variable inclusion companion of any logic is, at least, weakly paraconsistent. Moreover, if the law of non-contradiction fails in the original system, then it also fails in its left variable inclusion companion, hence in that case, the latter becomes strongly paraconsistent.
    \item For any logic induced by a Hilbert-style presentation, the restricted rules companion of it can be defined by keeping the same axioms and imposing variable inclusion restrictions on the rules of inference. This new logic is also, at least, weakly paraconsistent, and is strongly paraconsistent if the law of non-contradiction fails in the original logic.
    \item Suppose $\mathcal{S}=\langle\Fmbf,\vdash\rangle$ is a logic induced by a Hilbert-style presentation, and $\mathcal{S}^l=\langle\Fmbf,\lvari\rangle$ and $\mathcal{S}^{re}=\langle\Fmbf,\lrri\rangle$ are its left variable inclusion and restricted rules companions, respectively. Then $\lrri\,\subseteq\,\lvari$, but the converse is not always true. We have provided examples of logics for which the two companions are different.
    \item If the Deduction theorem holds in a logic $\mathcal{S}$, then it also holds in its left variable inclusion companion, $\mathcal{S}^l$. If the converse of the Deduction theorem holds in $\mathcal{S}$, then a restricted version of it holds in $\mathcal{S}^l$.
    \item If $\mathcal{S}$ is a logic induced by a Hilbert-style presentation, MP is a rule of inference of $\mathcal{S}$, and the Deduction theorem holds in $\mathcal{S}$, then $\mathcal{S}^l=\mathcal{S}^{re}$.
    \item There are paraconsistent logics that are neither the left variable inclusion companion nor the restricted rules companion of any logic.
\end{itemize}

The following are some possible directions for future work.

Instead of logics with algebraic counterparts featuring one contaminating element, one might investigate the logics corresponding to systems with finitely many or even infinitely many such elements and the algebraic issues arising out of this. Work in this line has recently been initiated in \cite{CiuniFergusonSzmuc2019}. The logical implications of such kind of algebraic studies is our area of interest.
 
It may be noted that the definition of the restricted rules companion of a logic can be extended to include logics presented via other proof systems. So, the following question comes up naturally. Can we mimic these techniques for logics which do not have Hilbert-style presentations? Work in this line has recently been initiated in \cite{Paoli2019}. 

%\phantomsection

\bibliographystyle{plain}
\bibliography{rrip-bib}

\end{document}